\documentclass[a4paper,11pt]{article}
\usepackage[utf8]{inputenc} 
\usepackage[T1]{fontenc} 
\usepackage[a4paper]{geometry} 
\usepackage{amsmath, amssymb, mathtools, amsthm}
\usepackage{mathtools} 
\usepackage{mathrsfs}  
\usepackage{graphicx} 
\usepackage{xcolor}   
\usepackage[english]{babel} 
\usepackage{float}
\usepackage{bbm}
\usepackage[title]{appendix}
\newcommand{\eps}{\varepsilon}
\usepackage{caption}
\usepackage{subcaption}
\usepackage{comment}
\usepackage{geometry}

\usepackage{enumitem} 
\usepackage[colorlinks,linkcolor=red,citecolor=blue]{hyperref}
\usepackage[sort]{cite}

\newcommand \commentout[1] {}

\newcommand{\R}{\mathbb{R}}
\newcommand{\N}{\mathbb{N}}

\newcommand {\Chi} {{\bf \raise 2pt \hbox{$\chi$}} }

\newcommand {\p}   {\partial}

\newcommand*{\dd}{\mathop{\kern0pt\mathrm{d}}\!{}}

\newcommand*{\DD}{\mathop{\kern0pt\mathrm{D}}\!{}}

\DeclarePairedDelimiter{\norm}{\|}{\|}

\theoremstyle{plain}
\newtheorem{theorem}{Theorem}[section]

\newtheorem{lemma}[theorem]{Lemma}
\newtheorem{proposition}[theorem]{Proposition}
\newtheorem{Ass}[theorem]{Assumption}
\theoremstyle{remark}
\newtheorem{remark}[theorem]{Remark}

\DeclareMathOperator{\DIV}{div}

\newcommand{\beq}{\begin{equation}}
\newcommand{\eeq}{\end{equation}}
\newcommand{\bea} {\begin{array}{rl}}
\newcommand{\eea} {\end{array}}
\newcommand{\bepa}{\left\{ \begin{array}{l}}
\newcommand{\eepa} {\end{array}\right.}
\newcommand{\diff}{\mathop{}\!\mathrm{d}}

\numberwithin{equation}{section}

\date{}

\begin{document}
\title{Nonlocal Cahn-Hilliard equation with degenerate mobility: Incompressible limit and convergence to stationary states}
\author{Charles Elbar\thanks{Sorbonne Universit\'{e}, CNRS, Universit\'{e} de Paris, Inria, Laboratoire Jacques-Louis Lions, F-75005 Paris, France } \thanks{email: charles.elbar@sorbonne-universite.fr}
\and Beno\^it Perthame\footnotemark[1]  \thanks{email: benoit.perthame@sorbonne-universite.fr}
\and Andrea Poiatti \thanks{Dipartimento di Matematica, Politecnico di Milano, Milano 20133, Italy}\thanks{email: andrea.poiatti@polimi.it}
\and Jakub Skrzeczkowski\thanks{Institute of Mathematics, Polish Academy of Sciences and Faculty of Mathematics, Informatics and Mechanics, University of Warsaw, Poland} 
\thanks{email: jakub.skrzeczkowski@student.uw.edu.pl}
}
\maketitle
\begin{abstract}
The link between compressible models of tissue growth and the Hele-Shaw free boundary problem of fluid mechanics has recently attracted a lot of attention. In most of these models, only repulsive forces and advection terms are taken into account. In order to take into account long range interactions, we include for the first time a surface tension effect by adding a nonlocal term which leads to the degenerate nonlocal Cahn-Hilliard equation, and study the incompressible limit of the system. The degeneracy and the source term are the main difficulties. Our approach relies on a new $L^{\infty}$ estimate obtained by De Giorgi iterations and on a uniform control of the energy despite the source term. We also prove the long-term convergence to a single constant stationary state of any weak solution using entropy methods, even when a source term is present. Our result shows that the surface tension in the nonlocal (and even  local) Cahn-Hilliard equation will not prevent the tumor from completely invading the domain. 
\end{abstract}

{\bf Conflict of interest statement:} The authors have no conflicts of interest to declare that are relevant to the content of this article.\\

{\bf Data availability statement:} Data availability is not applicable to this article as no new data were created or analyzed in this study.

\noindent{\makebox[1in]\hrulefill}\newline

\section{Introduction}

Nonlocal parabolic equations are commonly used to describe living tissues because cells experience two types of forces: repulsive and attractive. The repulsion arises at high volume fraction because, locally, cells occupy a non-vanishing volume, while cell adhesion and chemotaxis create attraction at long range, i.e., low densities~\cite{BCC2011}. These effects, as well as surface tension effects, can be considered by using the Cahn-Hilliard equation (see, e.g., \cite{MR4001523} for a review on these models). Our work is dedicated to the analysis of the nonlocal Cahn-Hilliard equation for long range interactions with a repulsive potential. More precisely  we are interested in 
%
 two results: the incompressible limit connecting mechanistic and free boundary descriptions of the tissue and the long-time asymptotics of equations. Concerning the first result, the main difficulty is that we lose any maximum principle and we have to rely on different arguments to obtain the same results concerning the incompressible limit. Concerning the convergence to the stationary state, we prove that it converges to a nonnegative constant which shows that the surface tension effect (modeled by the Cahn-Hilliard equation) is not strong enough to prevent the tumor from invading the entire domain. \\ 

\subsection{Mathematical setting}
Our settings is as follows: we let $\Omega$ be the $d=1,2,3$ dimensional flat torus in $\R^d$, which is particularly useful when treating nonlocal terms and we consider the equation

\begin{equation}
\label{nonlocal}
\begin{split}
\partial_t u -\text{div}(u\nabla\mu)=uG(p)\quad &\text{in }\Omega\times (0,T),\\
\mu=p+B_\varepsilon(u)\quad  &\text{in }\Omega\times (0,T),
\end{split}
\end{equation}
with the initial condition  $u(0)=u_0\geq 0 \text{ in }\Omega$ and $u(x,t) \geq 0$ represents the cell density. Here,  $B_\varepsilon$ denotes the nonlocal operator defined as
\begin{equation}\label{operatorB}
B_\varepsilon[u](x) = \frac{1}{\varepsilon^{2}}(u(x)-\omega_{\varepsilon}\ast u(x))=\frac{1}{\varepsilon^{2}}\int_{\Omega}\omega_{\varepsilon}(y)(u(x)-u(x-y)) \diff y
\end{equation}
for fixed $\varepsilon$ small enough (in order to be able to use \cite[Lemma C.1]{elbar-skrzeczkowski} and Lemma \ref{Poincar}) and $\omega_{\varepsilon}$ is a usual mollification kernel $\omega_{\varepsilon}(x)=\frac{1}{\varepsilon^{d}}\omega(\frac{x}{\varepsilon})$ with $\omega$ compactly supported in the unit ball of $\R^{d}$ satisfying
\begin{equation}
\int_{\R^{d}}\omega(y) \diff y =1, \quad  \omega(\cdot) \text{ is radial}. 
\label{as:omega}
\end{equation}
 The pressure and source term are defined, for $u\geq 0$, as 
 \begin{equation}
p=u^{\gamma}, \quad \gamma \geq 1, \qquad   G(p)=p_H-p,  
 \label{as:source}
\end{equation}
with  $p_{H}>0$ a constant called the homeostatic pressure, which is the threshold where cells begin to die, assuming that pressure produces an inhibitory effect on cell proliferation.\\
 
We comment the different terms appearing in the equation. First, $u$ is the density of tumor cells and can be thought of as being between 0 and 1. However, this fact is not easy to prove since the maximum principle does not hold here. Using a De Giorgi iteration technique, we can prove however that the bound holds with a small perturbation term which vanishes as $\gamma\to\infty$ (see Lemma \ref{DeG}). From the Cahn-Hilliard terminology, we refer to $\mu$ as the chemical potential, which is composed by two terms: one is the pressure $p=u^\gamma$ and the other is $B_\eps$, the approximation of the Laplace operator, which takes into account surface tension effects, see for instance~\cite{elbar-skrzeczkowski}.
Concerning the initial condition, we distinguish two sets of assumptions.
\begin{Ass}[Initial condition]\label{initial1}
We assume:\\
 (A) $0\leq u_0\leq p_H^\frac 1 \gamma$ for almost any $x\in \Omega$.

\end{Ass}
Note that the same assumption has already been considered, e.g., in \cite{david-phenotypic} and that it implies $u_0\in L^q(\Omega)$, for any $q\geq 1$, since $\Omega$ is of finite Lebesgue measure. For the single Section~\ref{section::incompressible},  we need additionally:

\begin{Ass}[Additional assumption for Section~\ref{section::incompressible}]\label{initial2}

We assume: \\
(B) There is $\gamma_0>0$ and $C=C(\gamma_0)>0$ such that $\Vert \Delta (u_0^{1+\gamma})\Vert_{L^1(\Omega)}
    +\Vert \nabla u_0\Vert_{L^2(\Omega)}+\Vert \Delta u_0\Vert_{L^1(\Omega)}\leq C,\quad \forall \gamma\geq\gamma_0$.
\end{Ass} 

\bigskip

System~\eqref{nonlocal} is associated with the energy $\mathcal{E}$ and entropy $\Phi$, respectively defined by 
\begin{align}
&\mathcal{E}(u):= \frac{1}{4\varepsilon^2}\int_\Omega\int_\Omega \omega_\varepsilon(y)\vert u(x)-u(x-y) \vert^2 \diff  y \diff x+\int_\Omega \dfrac{u^{\gamma+1}}{\gamma+1} \diff x \geq 0, \label{eq:energy}
\\
&\Phi(u):=\int_\Omega \left(\frac{u}{p_H^{\frac 1 \gamma}}\log\left(\frac{u}{p_H^{\frac 1 \gamma}}\right)-\frac{u}{p_H^{\frac 1 \gamma}}+1\right) \diff x \geq 0.\label{eq:entropy}
\end{align}
They formally satisfy the identities
\begin{equation}
\label{eq:energy_diss}
\frac{\diff}{\diff t}\mathcal{E}(u)+\int_{\Omega}u|\nabla\mu|^{2}=\int_{\Omega}u\mu G(p),
\end{equation}
\begin{align}
\nonumber \frac{\diff}{\diff t}\Phi(u)&+\frac{1}{2\varepsilon^2p_H^{\frac 1 \gamma}} \int_\Omega\int_\Omega \omega_\varepsilon(y)\vert \nabla u(x)-\nabla u(x-y) \vert^2\diff x \diff y\\&+\frac{1}{p_H^{\frac 1 \gamma}}\int_\Omega \dfrac{4\gamma}{(\gamma+1)^2}\left\vert \nabla \vert u\vert^{\frac{\gamma+1}{2}}\right\vert^2 \diff x -\frac{1}{p_H^{\frac 1 \gamma}}\int_\Omega u\log\left(\frac{u}{p_H^{\frac 1 \gamma}}\right)G(p)\diff x =0.
\label{entropy}
\end{align}
Moreover they provide us with direct a priori estimates, provided we can control the integral related to the source term in \eqref{eq:energy_diss}, which may change sign. 
Here, we assume that we have existence of solutions with regularity typical of the Cahn-Hilliard equation. We do not include the proof, since most of the a priori estimates are derived in Section \ref{sec:energy_estim}. For a rigorous proof of existence by means of an approximating scheme, we refer for instance to~\cite{elbar-skrzeczkowski}.
\begin{lemma}
\label{exist}
Let $u_0$ satisfy assumption  \eqref{initial1}. Then, for any $T>0$, there exist constants $C_0(T,\mathcal{E}({u_0}),\Phi(u_0),\varepsilon)$ and $C_1(T,\mathcal{E}({u_0}),\Phi(u_0),\varepsilon,\gamma)$ and a global weak solution $u$ such that,
\begin{align}
&u\geq 0\quad \text{a.e. in }\Omega\times(0,\infty),
\\
&u\in {C([0,\infty);(W^{1,r}(\Omega))^\prime)}\cap C_{weak}([0,\infty);L^{\gamma+1}(\Omega)),\quad r=\frac{(\gamma+1)(2\gamma+1)}{\gamma^2},\label{continuity}
\\
&\overline{u}(t):= \frac{1}{|\Omega|}\int_{\Omega} u(x,t) \diff x \leq p_H^\frac 1 \gamma\quad \forall t\geq 0, \label{mass}
\\
& 
\Vert u\Vert_{L^2(0,T;H^1(\Omega))} +\frac{1}{\gamma+1}\Vert u \Vert^{\gamma+1}_{L^\infty(0,T; L^{\gamma+1}(\Omega))}+\Vert u\Vert_{L^{2\gamma+1}(\Omega\times(0,T))}\leq C_0\label{unifm},
\\
& \Vert \partial_t u\Vert_{L^{2}(0,T;(W^{1,r}(\Omega))^\prime)}+\Vert \partial_t u\Vert_{L^{q^\prime}(0,T;(W^{1,q}(\Omega))^\prime)}\leq C_1, \label{dtes}
\end{align}
where $\frac 1 q +\frac{1}{q^\prime}=1$, $q=\frac{2(2\gamma+1)}{\gamma}$.
Moreover, for any $v\in W^{1,r}(\Omega)$, $u$ satisfies
\begin{align}
 \nonumber\langle\partial_t u,v\rangle_{(W^{1,r}(\Omega))',W^{1,r}(\Omega))}&+\int_\Omega u\nabla p\cdot \nabla v \diff x+\frac{1}{\varepsilon^2}\int_\Omega u\nabla u\cdot     \nabla v \diff x -\frac{1}{\varepsilon^2}\int_\Omega u(\nabla \omega_\varepsilon\ast u) \cdot \nabla v 
 \diff x \\&=\int_\Omega uG(p)v \diff x,  \qquad \text{for almost every } t >0,
 \label{weakf}
\end{align}
with $u(0)=u_0$ almost everywhere in $\Omega$. Here $\langle\cdot,\cdot\rangle_{(W^{1,r}(\Omega))',W^{1,r}(\Omega)}$ denotes the duality product between $(W^{1,r}(\Omega))'$ and $W^{1,r}(\Omega)$.
\end{lemma}  
\begin{remark}
\label{rem:weak_continuity}
The first continuity result follows from the bounds on $\p_t u$ in \eqref{dtes}.
Moreover, being $u\in L^\infty(0,T;L^{\gamma+1}(\Omega))$, we also deduce that $u\in C_{weak}([0,T];L^{\gamma+1}(\Omega))$, see \cite[Lemma II.5.9]{Boyer}. This clearly also implies that $\overline{u}=\frac{\int_\Omega u \diff x}{\vert\Omega\vert}\in C([0,T])$.
\label{continu}
\end{remark}

\subsection{The main results}
Our first result establishes the incompressible limit $\gamma \to \infty$ of the system \eqref{nonlocal} which links two descriptions of the tumor growth: mechanistic and free-boundary. The main mathematical novelty here is the nonlocality which makes it difficult to establish the uniform $L^{\infty}$ bound on~$u_{\gamma}$. To overcome this problem, we apply the De~Giorgi iterations, see Lemma \ref{DeG}, in the spirit of \cite{GGGsep, Poiatti2022}.

\begin{theorem} [Incompressible limit] \label{thm:incompressible}
Let $u_{\gamma}$ be a weak solution to ~\eqref{nonlocal} as defined in Lemma~\ref{exist} and initial datum satisfying Assumptions \ref{initial1}-\ref{initial2}.
Then, as $\gamma \to \infty$,  we have for all $T>0$, up to a (not relabeled) subsequence
\begin{align*}
&u_\gamma\overset{\ast} {\rightharpoonup}u_\infty\text{ in }L^\infty(\Omega\times(0,\infty)),\\&
u_\gamma\to u_\infty \text{ in }L^q(\Omega\times(0,T))\quad \forall q\in[2,+\infty), \\&
u_\gamma\rightharpoonup u_\infty \text{ in }L^2(0,T;H^1(\Omega)), \\&
\partial_t u_\gamma\rightharpoonup \partial_t u_\infty \text{ in }L^2(0,T;(H^1(\Omega))'),
\\&
p_\gamma\rightharpoonup p_\infty \text{ in }L^2(0,T;H^1(\Omega)),\\&
p_\gamma\to p_\infty \text{ in }L^r(\Omega\times(0,T)),\quad \forall r\in[2,3),
\end{align*}
where $u_{\infty}$ and $p_{\infty}$ satisfy in $\mathcal{D}'(\Omega\times [0,\infty))$
\begin{align}
&\partial_t u_\infty -\DIV(u_\infty \nabla(p_{\infty}+B_\varepsilon(u_\infty)))=u_\infty G(p_\infty),\label{incom}
\\
&p_\infty \Big(\Delta {p}_\infty+\frac{1}{2\varepsilon^2} \Delta u_\infty^2-\frac{1}{\varepsilon^2}\DIV(u_\infty(\nabla\omega_\varepsilon\ast u_\infty))+u_\infty G(p_\infty) \Big)=0,\label{eq:incompressible_pressure}
\\
& 0\leq u_\infty\leq 1, \quad p_\infty \geq 0, \quad p_{\infty}(1-u_{\infty})=0  \quad \text{almost everywhere in }  \Omega\times (0,\infty),
\label{eq:incompressible_graph}
\end{align}
with $u_\infty(0)=u_0$. Furthermore it holds $\langle \partial_t u_\infty, p_\infty\rangle_{(H^1(\Omega))',H^1(\Omega))}=0$ for almost any $t\in (0,\infty)$.
\end{theorem}

This theorem entails that in the limit $\gamma\to\infty$ we can consider the measurable set  $\Omega(t):=\{x\in\Omega:\ p_\infty(t)>0\}$,  where  $u_\infty=1$ by the graph relation \eqref{eq:incompressible_graph} so that it can be interpreted as the ‘tumor zone'. Note that it must hold
\begin{align*}
-\Delta {p}_\infty-\frac{1}{2\varepsilon^2} \Delta u_\infty^2=-\frac{1}{\varepsilon^2}\DIV(u_\infty(\nabla\omega_\varepsilon\ast u_\infty))+u_\infty G(p_\infty)\quad\text{ in } {\cal I}nt (\Omega(t)) 
\end{align*}
which yields a Hele-Shaw type equation.\\

Our second result is concerned with the convergence to stationary states. We distinguish two cases: when $G(p)=p_{H}-p$ and when $G(p)=0$. The main novelty lies in the first case which is not conservative and its proof requires a careful analysis of the entropy. We prove that as $t\to\infty$, the solution converges to the constant $p_H^\frac 1 \gamma > 0$, which shows that the surface tension is not strong enough to prevent the tumor from invading the entire domain. We have the 

\begin{theorem} [Long time behaviour] \label{thm:conv_stationary}
Let $u$ be a solution to \eqref{nonlocal} with fixed $\gamma\geq1$ in the sense of Lemma \ref{exist} and initial datum satisfying Assumption \ref{initial1} Then, if $u_0\not\equiv0$, there are two cases: 
\begin{itemize}
\item for $G(p)=p_{H}-p$, we have
\begin{equation}
\Vert u(t)-p_H^\frac 1 \gamma\Vert_{L^q(\Omega)}\to 0 \text{ as }t\to \infty,\quad \forall q\in [1,\gamma+1).
\label{conv}
\end{equation}
\item For $G(p)=0$ and $\gamma\geq1$, we have an exponential decay towards the mean value: there exists a constant $C=C(\Omega,\gamma,q,\Phi(u_0),\overline{u}_{0})$ such that
\begin{equation}
\Vert u(t)-\overline{u}_{0}\Vert_{L^q(\Omega)}\lesssim e^{-Ct} ,\quad \forall q\in [1,\gamma+1).
\label{conv2}
\end{equation}
\end{itemize}
\label{thm} 
\end{theorem}
There are two possibilities to prevent steady states to be constant. The first one is to consider different potentials than just repulsive ones like $u^{\gamma}$. The second possibility is to include an external force, which can be taken into account either by a generic force that acts directly on the cells like it was done in~\cite{elbar-perthame-skrzeczkowski}. One can also include the effects of nutriments and impose that the tumor cells die in the regions where there are no nutriments.

\begin{remark}
		In the case of nonzero source term $G$, Theorem \ref{thm:conv_stationary} implies that the constant solution $u \equiv 0$ is  unstable, whereas $p_H^\frac 1 \gamma$ is the only asymptotically stable equilibrium and \textit{any} weak solution $u$ (in the sense of Lemma \ref{exist}) departing from \textit{any} nonzero initial datum converges to $p_H^\frac 1 \gamma$. 
	\end{remark}

\subsection{Literature review}

\textbf{Incompressible limit for tumor model.} The incompressible limit connects two models of tumor growth: the compressible one studied in \cite{Perthame-Hele-Shaw} and the free-boundary one analysed in \cite{MR3695889}. Many substantial contributions followed \cite{Perthame-Hele-Shaw}, allowing nutrients \cite{MR4324293}, advection effects \cite{david2021incompressible,MR4331024,MR3942711}, additional structuring variable \cite{david-phenotypic}, congesting flows~\cite{2022arXiv220313709H}, two species \cite{MR4188329,MR4000848} or including additional surface tension effects via the degenerate Cahn-Hilliard equation \cite{elbar2021degenerate, elbar-perthame-skrzeczkowski}. \\ 

One major difficulty for establishing the incompressible limit is proving strong compactness of the gradient of the pressure. The main tool is the celebrated Aronson-Benilan estimate \cite{MR524760,MR670925}. The estimate has been recently readressed in \cite{bevilacqua2022aronson} but the generalization available there are not applicable for the pressure $p = u^{\gamma} + \frac{1}{\varepsilon^2} u$ as in our case cf. \cite[Theorem 4.1]{bevilacqua2022aronson}.
%
Another direct technique was developed in \cite{Liuguo} which is based on deducing strong convergence from a sort of energy equality. This is the strategy we follow in our proof. \\

\textbf{Nonlocal Cahn-Hilliard equation.} The nonlocal Cahn-Hilliard is a variant of the Cahn-Hilliard equation proposed to model dynamics of phase separation \cite{Cahn-Hilliard-1958}. While originally introduced in the context of material science, it is currently widely applied also in biology \cite{MR3104287, MR3902306, MR4199231}. The nonlocal equation
was obtained for the first time by Giacomin and Lebowitz as the limit of interacting particle systems \cite{MR1638739, MR1453735}. Their work can be considered as the first derivation of the Cahn-Hilliard equation up to a delicate limit from the nonlocal equation to the local one. The latter problem was in fact addressed only recently, first for the case of the constant mobility \cite{MR4198717, MR4093616, MR4408204, MR4248454} and finally for the case of degenerate mobility \cite{elbar-skrzeczkowski, carrillo2023degenerate}. Another derivation as a hydrodynamic limit of the Vlasov equation was proposed recently in \cite{elbar-mason-perthame-skrzeczkowski}, following \cite{takata2018simple}. In recent years nonlocal Cahn-Hilliard equation was also studied in couplings with other hydrodynamic models, like Navier-Stokes equations (see, e.g., \cite{FGG, Frig15, Frigeri, AGGP} and the references therein). Moreover, it has been adopted in many optimal control problems, we just mention \cite{PS,RS}.
\\

\textbf{Entropy dissipation methods and asymptotic analysis.} For establishing convergence to stationary states we use methods based on the entropy dissipation. In the simplest scenario, it can be applied to PDEs equipped with the entropy $\Phi$ which decreases with some dissipation 
$$
\partial_t \Phi(t) + \mathcal{D}\Phi(t) \leq 0, \qquad \mathcal{D}\Phi(t) \geq 0.
$$
Then, one tries to prove that the dissipation is bounded from below by the entropy $|\Phi(t)|^{\alpha} \leq \mathcal{D}\Phi(t)$ so that $\Phi(t) \to 0$ with an explicit convergence rate (exponential if $\alpha=1$ and polynomial if $\alpha>1$). Finally, by virtue of Csiszár-Kullback inequality, one deduces convergence in $L^1$. The last step requires conservation of mass which is not available when $G(p) \neq 0$ in \eqref{nonlocal}. We present the method in detail for the Cahn-Hilliard equation without the source term. Another method to obtain convergence to equilibrium, applied in the context of the Cahn-Hilliard equation, is via the \L ojasiewicz-Simon inequality~\cite{loja,abelswilke,LP2011}. This method cannot be applied here due to the degenerate mobility and the lack of separation property from the degenerate case $u=0$. 
\\



\section{Basic a priori estimates} \label{sec:energy_estim}
The energy/entropy structure usually provides a priori estimates on the solutions. However, in the case of a source term which may change sign, we first need to control their dissipation. 
Before tackling this problem, we first show a basic estimate which ensures the control of the mass of the system, uniformly in time. This estimate is useful to obtain a first $L^{\infty}_tL^{1}_x$ bound on the solution. Our proof of these estimates is somehow formal but can be carried out rigorously with an approximation scheme as, e.g., in \cite{elbar-skrzeczkowski}. These estimates are also fundamental to prove the existence Lemma~\ref{exist}.

\subsection{Control of the mass}
We recall that the total mass of the system is defined in~\eqref{mass} and we prove the corresponding bound.
\begin{proposition}[Mass control]\label{prop:mass_conserv} For all $t\ge 0$ we have $\overline{u}(t)\leq p_H^{\frac 1 \gamma}$. \end{proposition}
\begin{proof}
Integrating equation \eqref{nonlocal}, we obtain
\begin{align}
\frac{d}{dt} \overline{u}=\frac{1}{\vert \Omega\vert}\int_\Omega(u \, p_H-u^{\gamma+1}) \diff x.
\label{mass_control}
\end{align}
By  the H\"{o}lder inequality, 
$$
\int_{\Omega} u \diff x \leq \left(\int_{\Omega} u^{\gamma+1} \diff x \right)^{\frac{1}{\gamma+1}}
\vert \Omega \vert^\frac{\gamma}{\gamma+1}
$$
we deduce
\begin{align*}
\frac{d}{dt} \overline{u} &\leq \frac{p_H}{\vert\Omega\vert} \int_{\Omega}{u} \diff x-\left(\int_{\Omega}{u} \diff x\right)^{\gamma+1}\frac{1}{\vert \Omega\vert^{{\gamma+1}}}=p_H \overline{u}-\overline{u}^{1+\gamma},
\end{align*}
so that, with the Assumption~\ref{initial1} on the initial condition,  we can conclude~\eqref{mass}.
\end{proof}

\subsection{Energy and entropy estimates}\label{subsec:apriori_uniform_time}
We recall that the energy, the entropy as well as their dissipation have been defined in~\eqref{eq:energy}--\eqref{entropy}. We prove that, for a fixed time horizon $T$, we have the following inequalities.

\begin{proposition}[Control of the energy and entropy dissipation]\label{prop:energy/entropy}
The inequalities hold 
\begin{equation}\label{ent}
\begin{split}
 &\sup_{t\geq0}\Phi(u(t)) +
    \frac{1}{2\varepsilon^2p_H^{\frac 1 \gamma}}\int_0^\infty \int_\Omega\int_\Omega \omega_\varepsilon(y)\vert \nabla u(x)-\nabla u(x-y) \vert^2 \diff x \diff y \diff t\\&+\frac{1}{p_H^{\frac 1 \gamma}}\dfrac{4\gamma}{(\gamma+1)^2}\int_0^\infty\int_\Omega \left\vert \nabla \vert u\vert^{\frac{\gamma+1}{2}}\right\vert^2 \diff x \diff t + \int_0^{\infty} \int_\Omega u\log\left(\frac{u}{p_H^{\frac 1 \gamma}}\right)(p - p_H) \diff x \diff t \leq \Phi(u_0), 
\end{split}
\end{equation}
\begin{equation}\label{energy}
\begin{split}
\frac{d}{dt}\mathcal{E}(u)&+\frac{1}{2\varepsilon^2}\int_\Omega \omega_\varepsilon(y)(u(x)-u(x-y))(u^{\gamma+1}(x)-u^{\gamma+1}(x-y)) \diff x+\int_\Omega u\vert \nabla \mu\vert^2 \diff x\\&+\frac 1 2\int_\Omega u^{2\gamma+1} \diff x \leq C(\mathcal{E}(u)+1),  
\end{split}
\end{equation}
and thus  there exists $C(\mathcal{E}(u_0),T,\varepsilon)>0$ such that
$$
\sup_{t\in[0,T]}\mathcal{E}(u(t))+\frac 1 2\int_0^T\int_\Omega u^{2\gamma+1} \diff x \diff t \leq C(\mathcal{E}(u_0),T,\varepsilon).
$$
\end{proposition}
\begin{remark}
The above estimate in the energy $\mathcal{E}$ depends exponentially on the final time $T$. We improve this result to a global one 
in Proposition~\ref{prop:unif_energy_time}.
\end{remark} 
\begin{proof}
\textit{Control of the entropy}. Note that $\Phi(u)\geq 0$ by the inequality $x\log\left(\frac{x}{y}\right)-x+y\geq (\sqrt{x}-\sqrt{y})^2$ for $x\ge 0$ and $y>0$ ($x=\frac{u}{p_H^{\frac 1 \gamma}}$, $y=1$ here).
Then we have
\begin{align*}
-\int_\Omega u\log\left(\frac{u}{p_H^{\frac 1 \gamma}}\right)G(p) \diff x=-\int_\Omega u\log\left(\frac{u}{p_H^{\frac 1 \gamma}}\right)(p_H-p) \diff x \geq 0,
\end{align*}
since $u \geq 0$ and $x \mapsto \log x$ is increasing for $x \geq 0$.
Therefore, all the terms in the dissipation of entropy in \eqref{entropy} are nonnegative so that we can integrate in time and obtain \eqref{ent}, which clearly implies the control of the entropy \textit{independently} of $T$. 
\\ 

\textit{Energy control.} Turning to the energy $\mathcal{E}$, departing from \eqref{eq:energy_diss}, we observe that 
\begin{align*}
\int_\Omega uG(p)\mu \diff x =\int_\Omega u(p_H-p)p \diff x +\int_\Omega  u\,B_\varepsilon(u)(p_H-p) \diff x .
\end{align*}
The first term can be written as
\begin{align*}
\int_\Omega u(p_H-p)p \diff x &=-\int_{\Omega}u  p^{2} \diff x +\int_{\Omega} u  p_{H} p \diff x\\
&\le-\int_{\Omega}up^{2} \diff x +\frac{1}{2}\int_\Omega up^2 \diff x +\frac 1 2 p_H^2\int_\Omega u \diff x \le-\frac{1}{2}\int_{\Omega}up^{2} \diff x +C, 
\end{align*}
where we used the mass control \eqref{mass}. For the second term we have, by symmetry of $\omega$,
\begin{align*}
\int_\Omega u\,B_\varepsilon(u)(p_H-p) \diff x=& \frac{p_H}{2\varepsilon^2} \int_\Omega\int_\Omega \omega_\varepsilon(y)\vert u(x)-u(x-y) \vert^2 \diff x \diff y 
\\
&-\frac{1}{2\varepsilon^2}\int_\Omega\int_\Omega \omega_\varepsilon(y)(u(x)-u(x-y))(u^{\gamma+1}(x)-u^{\gamma+1}(x-y)) \diff x\diff y.
\end{align*}

All together, these inequalities give immediately \eqref{energy} and by the Gronwall lemma the energy control.
\end{proof}

\begin{remark}
In the limit $\gamma\to\infty$ it holds that $u_\infty\le 1$. This follows from the bound obtained with the energy in Proposition \ref{prop:energy/entropy} since 
$$
\Vert u_\infty \Vert_{L^\infty}  =   \lim_{\gamma \to \infty} \Vert u_\infty \Vert_{L^\gamma} \quad \text{and} \quad \Vert u \Vert_{L^{\gamma+1}}\leq C^{\frac 1 {\gamma+1}}(\gamma+1)^{\frac 1 {\gamma+1}} \to 1
$$
because of the weak convergence of $u^\gamma$. We refer for instance to~\cite{elbar2021degenerate,MR3695967}. However in the next section we obtain a better control on $u$ by the De Giorgi iteration method. 
\end{remark}
 
Now, we improve the local in time estimate on $\mathcal{E}$ to a global one, which is nontrivial due to the presence of the source term. Since Proposition \ref{prop:energy/entropy} gives the uniform control $\Phi (u(t))\leq \Phi(u_0)$ for any $t\geq0$, our aim is to control in a uniform way the energy $\mathcal{E}$ as well. 

\begin{proposition}[Uniform in time estimates for the energy]
\label{prop:unif_energy_time}
There exists a constant independent of time and $\gamma$ such that 
\begin{align}
\mathcal{E}(t)\leq C,\quad \forall t\geq 0.
\label{uniform}
\end{align}
\end{proposition}
\begin{proof}

Firstly, we estimate separately the two terms defining the energy in \eqref{eq:energy} using the entropy estimate \eqref{ent}. It immediately gives that, for a constant $C>0$ independent of time and $\gamma$,  for any sufficiently small $\varepsilon$, 
\begin{align}
\nonumber\int_t^{t+1}&\frac{1}{4\varepsilon^2}\int_\Omega\int_\Omega \omega_\varepsilon(y)\vert u(x)-u(x-y) \vert^2 \diff x \diff y \diff s \leq C(\Omega)\int_t^{t+1}\Vert u\Vert_{H^1(\Omega)}^2 \diff s
\\
&\leq C \int_t^{t+1}\frac{1}{4\varepsilon^2} \int_\Omega\int_\Omega \omega_\varepsilon(y)\vert \nabla u(x)-\nabla u(x-y) \vert^2 \diff x \diff y \diff s +C \int_t^{t+1}\Vert u \Vert_{L^1(\Omega)}^2 \diff s \label{parenergy}\\ 
\nonumber &\leq C+Cp_H^\frac 2 \gamma, \quad \forall t\geq 0.
\end{align}
where the second inequality follows from Equations \eqref{p_p} and \eqref{mass}.
\\

 Secondly, we control the remaining part of the energy $\mathcal{E}$, the one related to $\frac 1 {\gamma+1}\int_\Omega u^{\gamma+1} \diff x$. To this aim, we integrate Equation~\eqref{mass_control} in time  over $[t,t+1]$ and get
$$
\overline{u}(t+1)-\overline{u}(t)=\frac{1}{\vert \Omega\vert}\int_t^{t+1}\int_\Omega(up_H-u^{\gamma+1}) \diff x \diff s,
$$
so that, rearranging the terms, we control the second term of the energy as
$$
\frac{1}{\vert \Omega\vert}\int_t^{t+1}\int_\Omega u^{\gamma+1} \diff x \diff s \leq \frac{1}{\vert \Omega\vert}\int_t^{t+1}\int_\Omega up_H \diff x \diff s +\overline{u}(t)=
(p_H+1) \int_t^{t+1}\overline{u}(s) \diff s \leq C
$$
because  $\overline{u}(\cdot)\leq p_H^{\frac 1 \gamma}$ for any $t\geq0$ thanks to \eqref{mass}. This, together with \eqref{parenergy}, implies that 
\begin{align}
\int_t^{t+1}\mathcal{E}(s) \diff s \leq C,
\label{un}
\end{align}
with $C>0$ independent of $t\geq0$ and $\gamma$. 
\\

We may now  conclude the energy estimate. By Proposition \ref{prop:energy/entropy}, we have
$$
\frac{d}{dt}\mathcal{E}(u)\leq C (\mathcal{E}(u)+1).
$$
Using the Gronwall lemma, we obtain for all $t\geq 0$ and all $0\leq s \leq 1$,
$$
\mathcal{E}(u(t+1)) \leq C \, \mathcal{E}(u(t+s))+ C.
$$
Integrating in $s$ and using the bound \eqref{un}, we conclude the proof of Proposition~\ref{prop:unif_energy_time}. 

\end{proof}

\subsection{A control on $\partial_t u$}
\label{subsec::dtcontrol}
Here we prove the estimate on time derivative, which appears also in Lemma \ref{exist} and is used in the proof of Theorem \ref{thm:conv_stationary}.
\begin{proposition}
There exists $C=C(\gamma,T)>0$ such that the bounds hold
\begin{align}
&\Vert\partial_t u\Vert_{L^{q^\prime}(0,T;(W^{1,q}(\Omega))^\prime)}\leq C,\quad q={\dfrac{2(2\gamma+1)}{\gamma}}, \quad \frac 1 q +\frac 1 {q^\prime}=1,
    \label{d_tu}\\
&\Vert\partial_t u\Vert_{L^{2}(0,T;(W^{1,r}(\Omega))^\prime)}\leq C,\quad r=\frac{(\gamma+1)(2\gamma+1)}{\gamma^2}.
    \label{d_tu2}
\end{align}
 \label{partialtu}
\end{proposition}
\begin{proof}
For $\gamma\geq 1$ fixed, any $T>0$, any $\varphi\in L^q(0,T;W^{1,q}(\Omega))$,  we have
 \begin{align*}
     &\Big\vert \int_0^T  \langle\partial_t u,\varphi  \rangle_{(W^{1,q}(\Omega))',W^{1,q}(\Omega))} \diff t \Big\vert \\
     &\leq \left\vert \int_0^T\int_\Omega \gamma u^\gamma\nabla u\cdot \nabla \varphi \diff x \diff t \right\vert+\left\vert \int_0^T\frac 1 {\varepsilon^2}\int_\Omega u\nabla u\cdot \nabla \varphi \diff x \diff t \right\vert 
     \\
     &\qquad \quad +\left\vert \int_0^T\frac{1}{\varepsilon^2}\int_\Omega u(\nabla \omega_\varepsilon\ast u)\cdot \nabla \varphi \diff x \diff t \right\vert+\left\vert \int_0^T\int_\Omega u(p_H-u^\gamma)\varphi \diff x \diff t \right\vert \\
     &\leq \gamma^\frac 1 2\Vert u\Vert^{\frac{\gamma+1}{2(2\gamma+1)}}_{L^{2\gamma+1}(\Omega\times(0,T))}\Vert \gamma^\frac 1 2 u^{\frac{\gamma-1}{2}}\nabla u\Vert_{L^2(\Omega\times(0,T))} \Vert \nabla \varphi\Vert_{L^q(\Omega\times(0,T))}\\
     &\quad + C(\varepsilon)\Vert u\Vert_{L^{2\gamma+1}(\Omega\times(0,T))}\Vert \nabla u\Vert_{L^2(\Omega\times(0,T))} \Vert \nabla \varphi\Vert_{L^{\frac{2(2\gamma+1)}{2\gamma-1}}(\Omega\times(0,T))}
     \\
     &\quad +
     C(\varepsilon)\Vert u\Vert_{L^2(\Omega\times (0,T))}\Vert\nabla \omega_\varepsilon\ast u\Vert_{L^\infty(\Omega\times(0,T))}\Vert \nabla \varphi\Vert_{L^2(\Omega\times(0,T))}
     \\
     &\quad +C\Vert u\Vert_{L^{2\gamma+1}(\Omega\times(0,T))}\Vert p_H-u^\gamma\Vert_{L^2(\Omega\times(0,T))} \Vert  \varphi\Vert_{L^{\frac{2(2\gamma+1)}{2\gamma-1}}(\Omega\times(0,T))}
     \\
     & \leq 
     C(\varepsilon,T,\gamma) \Vert  \varphi\Vert_{L^q(0,T; W^{1,q}(\Omega))}.
     \end{align*}
     More precisely, to estimate the first and second terms,  we used the H\"{o}lder inequality with exponents $\frac{2(2\gamma+1)}{\gamma+1}$, $2$, $q$ and $2\gamma+1$, $2$ and $\frac{2(2\gamma+1)}{2\gamma-1}$, respectively. Then, $\Vert u\Vert^{\frac{\gamma+1}{2(2\gamma+1)}}_{L^{2\gamma+1}(\Omega\times(0,T))}$ is bounded due to \eqref{energy}, $\Vert \gamma^\frac 1 2 u^{\frac{\gamma-1}{2}}\nabla u\Vert_{L^2(\Omega\times(0,T))}$ is estimated by \eqref{ent} while the bound on $\nabla u$ follows from \eqref{ent} and nonlocal Poincar\'e inequality \eqref{p_p}. The fourth term is bounded in the same spirit. Concerning the third one, we simply estimate
     $$\Vert\nabla \omega_\varepsilon\ast u\Vert_{L^\infty(\Omega\times(0,T))}\leq \Vert \nabla \omega_\varepsilon\Vert_{L^\infty(\Omega)}\Vert u\Vert_{L^\infty(0,T;L^1(\Omega))}\leq C(\varepsilon)
     $$
     and use the estimate on the total mass \eqref{mass}. The final conclusion follows from the inequalities
     $q>2$ and $q\geq \frac{2(2\gamma+1)}{2\gamma-1}$ for any $\gamma\geq 1$.\\
     
     Concerning \eqref{d_tu2}, let $\varphi\in L^2(0,T;W^{1,r}(\Omega))$, with $r=\frac{(\gamma+1)(2\gamma+1)}{\gamma^2}$. Then,
      \begin{align*}
     &\left\vert \int_0^T\langle\partial_t u,\varphi\rangle_{(W^{1,r}(\Omega))',W^{1,r}(\Omega))}\diff t\right\vert\\ &\leq \left\vert \int_0^T\int_\Omega  u\nabla \mu\cdot \nabla \varphi \diff x \diff t\right\vert+\left\vert \int_0^T\int_\Omega u(p_H-u^\gamma)\varphi \diff x \diff t\right\vert\\&\leq \Vert u^\frac 1 2\Vert_{L^\infty(0,T;L^{2(\gamma+1)})}\Vert u^\frac 1 2\nabla \mu\Vert_{L^2(\Omega\times(0,T))}\Vert\nabla \varphi\Vert_{L^2(0,T;L^{\frac{2(\gamma+1)}{\gamma}}(\Omega))}
     \\
     &\qquad +p_H\Vert u\Vert_{L^\infty(0,T;L^{\gamma+1}(\Omega))}\Vert \varphi\Vert_{L^1(0,T;L^\frac{\gamma+1}{\gamma}(\Omega))}
     \\
     &\qquad +\Vert u\Vert_{L^\infty(0,T;L^{\gamma+1}(\Omega))}\Vert u\Vert_{L^{2\gamma+1}(\Omega\times(0,T))}\Vert \varphi\Vert_{L^{\frac{2\gamma+1}{2\gamma}}(0,T;L^r(\Omega))}\\&\leq C\Vert \varphi\Vert_{L^2(0,T;W^{1,r}(\Omega))},
     \end{align*}
     exploiting Proposition \ref{prop:energy/entropy}, recalling that $r\geq \frac{2(\gamma+1)}{\gamma}> \frac{\gamma+1}{\gamma}$ and $\frac{2\gamma+1}{2\gamma}\leq 2$. This concludes the proof.

\end{proof}

\section{Incompressible limit: proof of Theorem \ref{thm:incompressible}}
\label{section::incompressible}

We see the incompressible limit of \eqref{nonlocal} as the limit $\gamma\to \infty$ and the resulting problem turns out to be a free boundary problem of Hele-Shaw type. Concerning the techniques adopted here,  we  first show that, for any fixed $T>0$, $u_\gamma$ is bounded in $L^\infty(\Omega\times(0,T))$ by a quantity which converges to $1$ as $\gamma\to\infty$ (see \eqref{Linfty}). Due to the presence of the convolution term, which makes the equation nonlocal, we cannot apply any classical maximum principle, so that we need to resort to De Giorgi iterations, exploiting the fact that the equation is a second order differential equation.\\

With uniform estimates in $L^{\infty}$ at hand, we may apply standard energy estimates to gain sufficient regularity on the pressure $p_\gamma$, which we bounded in $L^3(\Omega\times(0,T))\cap L^2(0,T;H^1(\Omega))$ uniformly in $\gamma$. Then, one can obtain uniform controls in $L^\infty(0,T;L^1(\Omega))$ for $\partial_t u_\gamma$ and in $L^1(\Omega\times(0,T))$ for $\partial_t p_\gamma$, so as to deduce the strong convergence of $u_\gamma$ and~$p_\gamma$ in $L^2(\Omega\times (0,T))$ to some $u_\infty$ and $p_\infty$.  
With the help of these bounds, we are able to pass to the limit as $\gamma\to\infty$ in Equation~\eqref{nonlocal} and obtain Equations~\eqref{incom} and~\eqref{eq:incompressible_graph} for the limit concentration $u_\infty$.\\

In order to obtain more information on $u_\infty, p_\infty$, like complementarity conditions \eqref{eq:incompressible_pressure}, we need a stronger convergence for $\nabla p_\gamma$. The standard technique uses some control on $\Delta p$ thanks to the Aronson-B\'{e}nilan inequality (see, e.g., \cite{david2021incompressible}) which does
not apply here due to the higher-order term. In particular, the (formal) CH equation can be written as 
$$
\partial_t u_\gamma-\frac{\gamma}{\gamma+1}\Delta u^{\gamma+1}_\gamma- \frac{1}{2\varepsilon^2}\Delta u^2_\gamma +\frac{1}{\varepsilon^2}\text{div}(u_\gamma(\nabla \omega_\varepsilon\ast u)_\gamma)= u_\gamma G(p_\gamma),   
$$
and the extra term $\Delta u_\gamma^2$, independent of $\gamma$, appearing, and this prevents us from obtaining the Aronson-B\'{e}nilan inequality. Taking inspiration from \cite{david-phenotypic, Liuguo}, in the second part of the present section, we instead show the strong $L^2(\Omega\times(0,T))$ convergence of 
$$
\nabla \left(\frac{\gamma}{\gamma+1}u_\gamma^{\gamma+1}+\frac 1 {2\varepsilon^2} u^2_\gamma\right),
$$
which is shown to be enough to guarantee the  validity of the condition~\eqref{eq:incompressible_pressure}.

\subsection{An $L^\infty$ bound on $u_\gamma$}

\begin{lemma}
\label{DeG}
    Assume $0\leq u_0\leq p_H^{\frac 1 {\gamma}}$. For any $T>0$ there exists $\overline{\gamma}(T,\varepsilon)> 2$, explicitly computed as a function of $T$, such that
    \begin{align}
      0\leq u_\gamma\leq p_H^\frac 1 \gamma+\frac{2}{\sqrt[3]{\gamma}},\quad\text{a.e. on }\Omega\times(0,T),\quad \forall \gamma\geq \overline{\gamma}(T,\varepsilon). \label{Linfty}
    \end{align}
\end{lemma}
\begin{remark}
    Notice that the bound \eqref{Linfty} is useless to control the pressure, since $\left(p_H^\frac 1 \gamma+\frac{2}{\sqrt[3]{\gamma}}\right)^\gamma\to +\infty$ as $\gamma\to+\infty$.
\end{remark}
\begin{proof}
To simplify notations, we set $\xi:=\frac 1 {\sqrt[3]{\gamma}}$.
 The iterative scheme is as follows.  Let us consider the sequence 
$$
k_n=p_H^{\frac 1 \gamma }+ 2 \xi -\frac{\xi}{2^n}, \quad n\geq 0,
$$
and note that ${p}_H^{\frac 1 \gamma}+\xi \leq k_n< {p}_H^{\frac 1 \gamma}+ 2 \xi$.
Now we define the sequences
$$
A_n(t):=\{x\in\Omega:\ u_\gamma\geq k_n\},\quad y_n:=\int_0^T\int_{A_n(s)} \diff x \diff s .
$$
By testing the equation against $u_{n,\gamma}:=(u_{\gamma}-k_n)^+$, we immediately infer that 
\begin{align*}
\frac{1}{2}\dfrac{d}{dt}\Vert(u_\gamma-k_n)^+\Vert^2_{L^2(\Omega)}+&\int_\Omega \gamma u_\gamma^\gamma\vert \nabla (u_\gamma-k_n)^+\vert^2 \diff x +\int_\Omega u_\gamma\nabla B_\varepsilon(u_\gamma)\cdot \nabla (u_\gamma-k_n)^+ \diff x \\&=\int_{\Omega} u_\gamma G(p_\gamma)(u_\gamma-k_n)^+ \diff x.
\end{align*}
By the definition of $B_\varepsilon$, we get
$$
\int_\Omega u_\gamma\nabla B_\varepsilon(u_\gamma)\cdot \nabla (u_\gamma-k_n)^+ \diff x =\frac{1}{\varepsilon^2}\int_\Omega {u_\gamma}\vert \nabla u_{n,\gamma}\vert^2 \diff x -\frac{1}{\varepsilon^2}\int_\Omega u_\gamma(\nabla\omega_\varepsilon\ast u_\gamma)\cdot \nabla u_{n,\gamma} \diff x.
$$
The first term is nonnegative. For the second we use that $\gamma>2$, $\Vert u_\gamma\Vert_{L^\infty(0,\infty;L^1(\Omega))}\leq {p}_H^{\frac 1 \gamma}\vert \Omega\vert$ thanks to \eqref{mass} as well as $u_{\gamma} \geq k_n \geq p_H^{1/\gamma} + \xi$ on $A_n$ to obtain
\begin{align*}
&\frac{1}{\varepsilon^2}\int_\Omega u_\gamma(\nabla\omega_\varepsilon\ast u_\gamma)\cdot \nabla u_{n,\gamma} \diff x \leq \frac{1}{\varepsilon^4}\frac{1}{2\gamma}\int_{A_n} u^{2-\gamma}_\gamma\vert \nabla\omega_\varepsilon\ast u_\gamma\vert^2 \diff x +\frac{\gamma}{2}\int_\Omega u^{\gamma}_\gamma\vert \nabla u_{n,\gamma}\vert ^2 \diff x
\\&\leq \frac{1}{\varepsilon^4}\frac 1 {2\gamma \left({p}_H^\frac 1 \gamma +\xi\right)^{\gamma-2}}\Vert \nabla\omega_\varepsilon\ast u_\gamma\Vert_{L^\infty(\Omega)}^{2}\int_{A_n} \diff x +\frac{\gamma}{2}\int_\Omega u_\gamma^{\gamma}\vert \nabla u_{n,\gamma}\vert ^2 \diff x
\\
&\leq \frac{1}{\varepsilon^4}\frac 1 {2\gamma \left({p}_H^\frac 1 \gamma +\xi\right)^{ {\gamma-2}}}\Vert \nabla\omega_\varepsilon\Vert_{L^\infty(\Omega)}^{2}\Vert u_\gamma\Vert_{L^1(\Omega)}^{2}\int_{A_n} \diff x +\frac{\gamma}{2}\int_\Omega u_\gamma^{\gamma}\vert \nabla u_{n,\gamma}\vert ^2 \diff x \\&
\leq \frac {C_1} {\gamma \left({p}_H^\frac 1 \gamma +\xi\right)^{ {\gamma-2}}}\int_{A_n} \diff x +\frac{\gamma}{2}\int_\Omega u_\gamma^{\gamma}\vert \nabla u_{n,\gamma}\vert ^2 \diff x ,
\end{align*}
where $C_1>0$ is a constant that depends on $\varepsilon$ but not on $T$ and $\gamma$.
Moreover, since on $A_n$ we have $u_\gamma\geq {p}_H^{\frac 1 \gamma}$ and thus $p_\gamma\geq p_H$, we immediately infer that
\begin{align*}
\int_{\Omega} u_\gamma G(p_\gamma)(u_\gamma-k_n)^+ \diff x \leq 0. 
\end{align*}
We can then sum up the results above to obtain 
\begin{align*}
\frac{1}{2}\dfrac{d}{dt}\Vert(u_\gamma-k_n)^+\Vert^2_{L^2(\Omega)}+\frac{\gamma}{2}\int_\Omega u_\gamma^\gamma\vert \nabla (u_\gamma-k_n)^+\vert^2 \diff x  \leq \frac {C_1} {\gamma \left({p}_H^\frac 1 \gamma +\xi\right)^{ {\gamma-2}}}\int_{A_n} \diff x,
\end{align*}
which also implies, since on $A_n(t)$ we have $u^\gamma_\gamma\geq \left({p}_H^\frac 1 \gamma +\xi\right)^{ {\gamma}}$,
\begin{align*}
\frac{1}{2}\dfrac{d}{dt}\Vert(u_\gamma-k_n)^+\Vert^2_{L^2(\Omega)}+\frac{\gamma \left({p}_H^\frac 1 \gamma +\xi\right)^{ {\gamma}}}{2}\int_\Omega \vert \nabla (u_\gamma-k_n)^+\vert^2 \diff x  \leq \frac {C_1} {\gamma \left({p}_H^\frac 1 \gamma +\xi\right)^{ {\gamma-2}}}\int_{A_n} \diff x.
\end{align*}
It is now clear that,  
\begin{align}
&\sup_{t\in[0,T]}\Vert u_{n,\gamma}(t)\Vert^2_{L^2(\Omega)}\leq \frac {2C_1} {\gamma \left({p}_H^\frac 1 \gamma +\xi\right)^{ {\gamma-2}}}y_n=:Z_n,\label{e1}\\& \gamma\left({p}_H^\frac 1 \gamma +\xi\right)^{ {\gamma}}\int_0^T\int_\Omega \vert \nabla (u_\gamma-k_n)^+(s)\vert^2 \diff x \diff s \leq Z_n.
\label{estimates}
\end{align}
where we used, by the assumptions on the initial conditions, $\Vert u_{n,\gamma}(0)\Vert_{L^2(\Omega)}=0$. 
Now for any $t$ and for almost any $x\in A_{n+1}(t)$, we get
\begin{align*}
&u_{n,\gamma}(x,t)=
\underbrace{u_\gamma(x,t)-\left[{p}_H^{\frac 1 \gamma}+ 2 \xi -\frac{\xi}{2^{n+1}}\right]}_{u_{n+1,\gamma}(x,t)\geq 0}+  \xi\left[\frac{1}{2^{n}}-\frac{1}{2^{n+1}}\right]\geq \frac{\xi}{2^{n+1}}.
\end{align*}
Then we have
\begin{align*}
\int_0^T\int_{\Omega}\vert u_{n,\gamma}\vert^3\diff x \diff s\geq \int_0^T\int_{A_{n+1}(s)}\vert u_{n,\gamma}\vert^3\diff x \diff s\geq \left(\frac{\xi}{2^{n+1}}\right)^3\int_0^T\int_{A_{n+1}(s)}\diff x \diff s=\left(\frac{\xi}{2^{n+1}}\right)^3y_{n+1}.
\end{align*}
Then we have
\begin{align}
&\left(\frac{\xi}{2^{n+1}}\right)^3y_{n+1}\leq   \left(\int_0^T\int_{\Omega}\vert 
u_{n,\gamma}\vert^{\frac{10}{3}}\diff x \diff s\right)^{\frac{9}{10}}\left(\int_0^T\int_{A_n(s)}\diff x \diff s\right)^{\frac{1}{10}}.
\label{est1}
\end{align}
For the sake of clarity we now present the argument in the case $d=3$, but it can be easily adapted to any dimension $d\geq1$.
We recall that by a variant of the three-dimensional Sobolev-Gagliardo-Nirenberg inequality (see, e.g.,\cite[Ch.9]{Brezis2010FunctionalAS}) we get
\begin{align*}
&\| v - \overline{v}\|_{L^\frac{10}{3}(\Omega)}\leq {C_G}\, \|v\|_{L^2(\Omega)}^{\frac{2}{5}}\|\nabla v\|_{L^2(\Omega)}^{\frac{3}{5}} \quad \forall  v \in H^1(\Omega),
\end{align*}
with ${C_G}>0$ depending only on $\Omega$. Therefore,
\begin{align*}
&\int_0^T\int_{\Omega}\vert u_{n,\gamma}\vert^{\frac{10}{3}}\diff x \diff s\leq
2^{\frac 7 3}\int_0^T\int_{\Omega}\vert u_{n,\gamma}-\overline{u}_{n,\gamma}\vert^{\frac{10}{3}}\diff x \diff s+2^{\frac 7 3}\int_0^T\int_{\Omega}\vert \overline{u}_{n,\gamma}\vert^{\frac{10}{3}}\diff x \diff s\\&\leq C_G2^{\frac 7 3}\int_{0}^T\Vert \nabla u_{n,\gamma}\Vert^2_{L^2(\Omega)}\Vert u_{n,\gamma}\Vert^{\frac{4}{3}}_{L^2(\Omega)}\diff s+\frac{2^{\frac 7 3}}{\vert \Omega\vert ^{\frac 7 3}}\int_0^T\Vert u_{n,\gamma}\Vert_{L^1(\Omega)}^{\frac {10} 3}\\&\leq C_G2^{\frac 7 3}\int_{0}^T\Vert \nabla u_{n,\gamma}\Vert^2_{L^2(\Omega)}\Vert u_{n,\gamma}\Vert_{L^2(\Omega)}^{\frac{4}{3}}\diff s+\frac{2^{\frac 7 3}}{\vert \Omega\vert ^{\frac 2 3}}\int_0^T\Vert u_{n,\gamma}\Vert_{L^2(\Omega)}^{\frac {10} 3},
\end{align*}
so that by \eqref{e1} and \eqref{estimates} we immediately infer
\begin{align*}
\int_0^T \int_{\Omega} &\vert u_{n,\gamma}\vert^{\frac{10}{3}}\diff x \diff s\\&\leq \frac{C_G}{\gamma\left({p}_H^\frac 1 \gamma +\xi\right)^{ {\gamma}}}\sup_{t\in [0,T]}\Vert u_{n,\gamma}(t)\Vert^{\frac{4}{3}}_{L^2(\Omega)}2^{\frac 7 3}\gamma\left({p}_H^\frac 1 \gamma +\xi\right)^{ {\gamma}}\int_{0}^T\Vert \nabla u_{n,\gamma}(s)\Vert^2_{L^2(\Omega)}\diff s
\\
& \qquad +\sup_{t\in[0,T]}\Vert u_{n,\gamma}(t)\Vert_{L^2(\Omega)}^{\frac {10} 3}\frac{T2^{\frac 7 3}}{\vert \Omega\vert ^{\frac 2 3}}
\\
& \leq \frac{2^{\frac 7 3}C_G}{\gamma\left({p}_H^\frac 1 \gamma +\xi\right)^{ {\gamma}}} Z_n^{\frac{5}{3}}+\frac{T2^{\frac 7 3}}{\vert \Omega\vert ^{\frac 2 3}}Z_n^{\frac{5}{3}}
\leq \frac{T2^{\frac {10} 3}}{\vert \Omega\vert ^{\frac 2 3}} Z_n^{\frac 5 3}\leq \frac {C_2T} {\gamma^{\frac 5 3} \left({p}_H^\frac 1 \gamma +\xi\right)^{ {\frac 5 3\gamma-\frac {10} 3}}} y_n^{\frac 5 3},
\end{align*}
with $C_2=C_2(\varepsilon)>0$. Note that we have assumed $\gamma$ sufficiently large, say $\gamma\geq \gamma_0(T)>2$ so that
\begin{align}
\gamma\left({p}_H^\frac 1 \gamma +\xi\right)^\gamma\geq \dfrac{C_G\vert \Omega\vert^\frac 2 3}{T}.
    \label{AA}
\end{align}
 Coming back to \eqref{est1}, we get
\begin{align*}
    \left(\frac{\xi}{2^{n+1}}\right)^3 y_{n+1}\leq \frac{C_2^{\frac 9 {10}} T^{\frac{9}{10}}}{\gamma^{\frac {3} 2}\left({p}_H^\frac 1 \gamma +\xi\right)^{\frac 3 2\gamma-3}}y_n^{\frac 8 5},
\end{align*}
i.e., recalling the definition of $\xi$,
\begin{align*}
     y_{n+1}\leq \frac{2^{3n+3}C_2^{\frac 9 {10}} T^{\frac{9}{10}}}{\gamma^{\frac {1} 2}\left({p}_H^\frac 1 \gamma +\xi\right)^{\frac 3 2\gamma-3}}y_n^{\frac 8 5}.
\end{align*}
Due to Lemma \ref{conva} with $b=2^3>1$, $C=\frac{2^{3}C_2^{\frac 9 {10}} T^{\frac{9}{10}}}{\gamma^{\frac {1} 2}\left({p}_H^\frac 1 \gamma +\xi\right)^{\frac 3 2\gamma-3}}>0$, $\epsilon=\frac{3}{5}$, we get that ${y}_n\to 0$ if 
\begin{align}
{y}_0\leq C^{-\frac{5}{3}}b^{-\frac{25}{9}} \iff y_0\leq 2^{-\frac{25}{3}}\frac{\gamma^{\frac 5 6}\left({p}_H^\frac 1 \gamma +\xi\right)^{\frac 5 2\gamma-5}}{ 2^5C_2^{\frac 3 2}T^{\frac 3 2}}
\label{lastbis}
\end{align}
As we have
$
y_0\leq T\vert \Omega\vert,
$
it is enough to ask for $\gamma$ sufficiently large, say $\gamma\geq \gamma_1(T,\varepsilon)\geq \gamma_0(T)$ such that
$$
T^\frac{5}{2}\vert \Omega\vert\leq2^{-\frac{25}{3}}\frac{\gamma^{\frac 5 6}\left({p}_H^\frac 1 \gamma +\xi\right)^{\frac 5 2\gamma-5}}{ 2^5C_2^{\frac 3 2}},
$$
i.e.,
$$
y_0\leq T\vert \Omega\vert\leq 2^{-\frac{25}{3}}\frac{\gamma^{\frac 5 6}\left({p}_H^\frac 1 \gamma +\xi\right)^{\frac 5 2\gamma-5}}{ 2^5C_2^{\frac 3 2}T^{\frac 3 2}}.
$$
This way $y_n\to 0$ as long as $\overline{\gamma}(T,\varepsilon)\geq\gamma_1(T,\varepsilon)$ and any $\gamma\geq \overline{\gamma}(T,\varepsilon)$. 
\end{proof}
\subsection{Higher-order regularity results, uniformly in $\gamma$.}
\begin{lemma}   \label{ests}
    For any $T>0$ there exists $C=C(T,\varepsilon)>0$ such that 
    \begin{align}
    &\Vert p_\gamma\Vert_{L^2(0,T;H^1(\Omega))}+\Vert  p_\gamma\Vert_{L^3(\Omega\times(0,T))}\leq C,\qquad \forall \gamma> 1,\label{est_p}
    \\
    &\Vert \partial_t u_\gamma\Vert_{L^2(0,T;(H^1(\Omega))^\prime)}\leq C,\quad \forall \gamma\geq\overline{\gamma}(T,\varepsilon) \qquad \text{(see } \eqref{Linfty}) ,\label{dtu1}
    \\
    & \Vert\partial_t u_\gamma\Vert_{L^\infty(0,T;L^1(\Omega))}+\Vert \partial_t p_\gamma\Vert_{L^1( \Omega)\times (0,T)}\leq C,\qquad \forall \gamma\geq 1.  \label{bb}
    \end{align}
\end{lemma}
\begin{proof}
The arguments of the proof are often written formally for simplicity, but can be easily made rigorous in a suitable approximating scheme.
Note that from  Proposition \ref{prop:energy/entropy} we are able to deduce that 
\begin{align}
p_\gamma=u_\gamma^\gamma\in L^{2+\frac{1}{\gamma}}(\Omega\times(0,T))\hookrightarrow L^{2}(\Omega\times(0,T))
\label{fundamental}
\end{align}
uniformly in $\gamma$. Thus, to prove the $H^1$ bound in \eqref{est_p},  we only need to find an estimate for the gradient of $p_\gamma$.
Let us consider $\int_\Omega u_\gamma^\gamma \diff x$ and compute its time derivative: from \eqref{nonlocal} we infer
\begin{align}
&\frac{d}{dt}\int_\Omega u_\gamma^\gamma \diff x = \int_\Omega \gamma u_\gamma^{\gamma-1}\partial_t u_\gamma \diff x \nonumber\\&= -\gamma \int_\Omega u_\gamma\nabla(u^{\gamma-1})\nabla (u_\gamma^\gamma)\diff x-\gamma\int_\Omega u_\gamma\nabla(u_\gamma^{\gamma-1})\cdot\nabla B_\varepsilon(u_\gamma) \diff x+\gamma\int_\Omega u_\gamma(p_H-p_\gamma)u_\gamma^{\gamma-1}\diff x\nonumber
\\&=-\gamma^2(\gamma-1)\int_\Omega u_\gamma^{2\gamma-2}\vert \nabla u_\gamma\vert^2 \diff x-\gamma\int_\Omega u_\gamma\nabla(u_\gamma^{\gamma-1})\cdot\nabla B_\varepsilon(u_\gamma) \diff x+\gamma\int_\Omega u_\gamma(p_H-p_\gamma)u_{\gamma}^{\gamma-1}\diff x\nonumber\\&
=-(\gamma-1)\int_\Omega \left\vert \nabla (u_\gamma^\gamma)\right\vert^2\diff x-\gamma(\gamma-1)\int_\Omega u_\gamma^{\gamma-1}\nabla u_\gamma\cdot\nabla B_\varepsilon(u_\gamma) \diff x+\gamma\int_\Omega u_\gamma(p_H-p_\gamma)u_\gamma^{\gamma-1}\diff x.
\label{est2}
\end{align}
Due to \eqref{operatorB}, we have 
\begin{align*}
&\gamma(\gamma-1)\int_\Omega u_\gamma^{\gamma-1}\nabla u_\gamma\cdot\nabla B_\varepsilon(u_\gamma) \diff x\\&= \frac{1}{\varepsilon^2}{\gamma(\gamma-1)}\int_\Omega u_\gamma^{\gamma-1}\vert \nabla u_\gamma\vert^2-\frac{1}{\varepsilon^2}{\gamma(\gamma-1)}\int_\Omega u_\gamma^{\gamma-1}\nabla u_\gamma\cdot \nabla(\omega_\varepsilon\ast u_\gamma) \diff x
\end{align*}
By the Young inequality, recalling that $\Vert u_\gamma\Vert_{L^\infty(0,T;L^2(\Omega))}\leq C$ (due to the uniform bound on the energy and the non-local Poincar\'e inequality \eqref{eq:poincare_L2}, see Lemma \ref{exist}), we get  
\begin{align*}
&\frac{1}{\varepsilon^2}{\gamma(\gamma-1)}\int_\Omega u_\gamma^{\gamma-1}\nabla u_\gamma\cdot \nabla(\omega\ast u_\gamma) \diff x\\&\leq \frac{\gamma^2(\gamma-1)}{2}\int_\Omega u_\gamma^{2\gamma-2}\vert \nabla u_\gamma\vert^2\diff x +\frac{1}{\varepsilon^4}\frac{\gamma-1}{2}\int_\Omega \vert \nabla\omega_\varepsilon\ast u_\gamma\vert^2 \diff x\\&\leq \frac{\gamma^2(\gamma-1)}{2}\int_\Omega u_\gamma^{2\gamma-2}\vert \nabla u_\gamma\vert^2\diff x+\frac{1}{\varepsilon^4}\frac{\gamma-1}{2}\Vert \nabla\omega_\varepsilon\Vert^2_{L^1(\Omega)}\Vert {u}_\gamma\Vert^2_{L^2(\Omega)}\\&\leq \frac{(\gamma-1)}{2}\int_\Omega \vert \nabla (u_\gamma^{\gamma})\vert^2\diff x+\frac{C(T,\varepsilon)(\gamma-1)}{2}.
\end{align*}
The last term in \eqref{est2} can be controlled by
$$
\int_\Omega u_\gamma(p_H-p_\gamma)\gamma u_\gamma^{\gamma-1}\diff x\leq \int_\Omega (p_H-p_\gamma)\gamma u_\gamma^{\gamma}\diff x \leq \int_{p_\gamma\leq p_H} (p_H-p_\gamma)\gamma u^{\gamma}_\gamma \diff x\leq \gamma\vert \Omega\vert  p_H^2 .
$$

Therefore, from \eqref{est2} we get
\begin{align}
\frac{1}{\gamma-1}\frac{d}{dt}\int_\Omega u_\gamma^\gamma \diff x+\frac{1}{2}\int_\Omega \left\vert \nabla (u_\gamma^\gamma)\right\vert^2\diff x\leq \frac{C}{2}+C\frac{\gamma}{\gamma-1}\leq C(T,\varepsilon),
\label{grad}
\end{align}
showing that $\nabla u^\gamma\in L^2(\Omega\times (0,T))$ uniformly in $\gamma$ so that
$u_\gamma^\gamma$ is bounded uniformly in $L^2(0,T;H^1(\Omega))$. We now need a similar estimate to show the $L^3$  bound in \eqref{est_p}. We have 
\begin{align}
\frac{1}{2\gamma}&\frac{\diff}{\diff t}\int_\Omega u_\gamma^{2\gamma} \diff x = \int_\Omega  u_\gamma^{2\gamma-1}\partial_t u_\gamma \diff x \nonumber
\\
&= - \int_\Omega u_\gamma\nabla(u_\gamma^{2\gamma-1})\nabla (u_\gamma^\gamma)\diff x- \int_\Omega u_\gamma\nabla(u_\gamma^{2\gamma-1})\cdot\nabla B_\varepsilon(u_\gamma) \diff x+\int_\Omega u_\gamma(p_H-p)u_\gamma^{2\gamma-1}\diff x\nonumber
\\
&=-\gamma(2\gamma-1)\int_\Omega u_\gamma^{3\gamma-2}\vert \nabla u_\gamma\vert^2 \diff x-\frac{1}{\varepsilon^2}(2\gamma-1)\int_\Omega u_\gamma^{2\gamma-1}\vert\nabla u_\gamma\vert^2 \diff x\nonumber
\\
& \quad +\frac{1}{\varepsilon^2}{(2\gamma-1)}\int_\Omega u_\gamma^{2\gamma-1}\nabla u_\gamma\cdot (\nabla \omega_\varepsilon\ast u_\gamma) \diff x+ \int_\Omega (p_H-p_\gamma)u_\gamma^{2\gamma}\diff x.
\label{est3}
\end{align}
Notice now that, by Young's inequality for convolutions and after integration by parts,
\begin{align*}
\frac{1}{\varepsilon^2}{(2\gamma-1)}&\int_\Omega u^{2\gamma-1}_\gamma\nabla u_\gamma\cdot (\nabla \omega_\varepsilon\ast u_\gamma) \diff x
\\
&=\frac{1}{\varepsilon^2} \frac{2\gamma-1}{2\gamma}\int_\Omega \nabla u^{2\gamma}_\gamma\cdot (\nabla\omega_\varepsilon\ast u_\gamma)\diff x
=\frac{1}{\varepsilon^2}\frac{1-2\gamma}{2\gamma}\int_\Omega (\Delta\omega_\varepsilon\ast u_\gamma)u_\gamma^{2\gamma}\diff x
\\
&\leq \frac{1}{\varepsilon^2} \frac{2\gamma-1}{2\gamma}\Vert \Delta \omega_\varepsilon\Vert_{L^\infty(\Omega)}\Vert u_\gamma\Vert_{L^1(\Omega)}\int_\Omega p_\gamma^2\diff x\leq C(\varepsilon) \frac{2\gamma-1}{2\gamma}\int_\Omega p_\gamma^2\diff x.
\end{align*}
We thus get, integrating \eqref{est3} over $(0,T)$ and using the $L^2$ bound on $p_\gamma$ in~\eqref{fundamental},
\begin{align*}
\frac 1{2\gamma}\int_\Omega p_\gamma^2(T) \diff x &+{\gamma(2\gamma-1)}\int_0^T\int_\Omega u_\gamma^{3\gamma-2}\vert \nabla u_\gamma\vert^2 \diff x \diff s
+\frac{2\gamma-1}{\varepsilon^2}\int_0^T\int_\Omega u_\gamma^{2\gamma-1}\vert\nabla u_\gamma\vert^2 \diff x \diff s
\\
&+\int_0^T\int_\Omega p_\gamma^{3\gamma}\diff x \diff s\leq C(T,\varepsilon)\left(\dfrac{2\gamma-1}{2\gamma}+1\right)\leq C(T,\varepsilon),
\end{align*}
where we used that $\frac{2\gamma-1}{2\gamma}\to 1$ as $\gamma\to \infty$. From this we deduce the uniform $L^3$ estimate in~\eqref{est_p}.
\\

To prove \eqref{dtu1}, we compute for any $\varphi\in L^2(0,T;H^1(\Omega))$ and $\gamma\geq \overline{\gamma}$, with $\overline{\gamma}(T)$ as in \eqref{Linfty} 
\begin{align*}
     &\left\vert \int_0^T\langle \partial_t u_\gamma,\varphi \rangle \diff t\right\vert\leq \left\vert \int_0^T\int_\Omega u_\gamma\nabla p_\gamma\cdot \nabla \varphi \diff x \diff t \right\vert+\left\vert \int_0^T\frac 1 {\varepsilon^2}\int_\Omega u_\gamma\nabla u_\gamma\cdot \nabla \varphi \diff x \diff t \right\vert \\&+\left\vert \int_0^T\frac{1}{\varepsilon^2}\int_\Omega u_\gamma(\nabla \omega_\varepsilon\ast u_\gamma)\cdot \nabla \varphi \diff x \diff t \right\vert+\left\vert \int_0^T\int_\Omega u_\gamma(p_H-u_\gamma^\gamma)\varphi \diff x \diff t \right\vert.
     \end{align*}
Note that by Young's inequality for convolutions we get
\begin{align*}
   &\Vert \nabla \omega_\varepsilon\ast u_\gamma\Vert_{L^2(\Omega\times(0,T))}^2\leq 
   \int_0^T\Vert \nabla\omega_\varepsilon\ast u_\gamma\Vert_{L^\infty(\Omega)}^2\vert \Omega\vert \diff t\leq \vert \Omega\vert \Vert \nabla\omega_\varepsilon\Vert_{L^\infty(\Omega)}^2\int_0^T\Vert u_\gamma\Vert_{L^1(\Omega)}^2\diff t
   \\&\leq \vert \Omega\vert \Vert \nabla\omega_\varepsilon\Vert_{L^\infty(\Omega)}^2\int_0^T\left(\int_\Omega p_H^\frac 1 \gamma \diff x\right)^2\diff t\leq \vert \Omega\vert^3 \Vert \nabla\omega_\varepsilon\Vert_{L^\infty(\Omega)}^2 p_H^\frac 2 \gamma T\leq C(\varepsilon,T), 
\end{align*}
by \eqref{mass}. Therefore, by Lemma \ref{exist}, \eqref{Linfty} and \eqref{est_p}    
     \begin{align*}
    \left\vert \int_0^T\langle \partial_t u_\gamma,\varphi \rangle \diff t\right\vert \leq\,& \Vert u_\gamma\Vert_{L^\infty(\Omega\times(0,T))}\Vert \nabla p_\gamma\Vert_{L^2(0,T;L^2(\Omega))}\Vert \nabla\varphi\Vert_{L^2(0,T;L^2(\Omega))}\\&+C(\varepsilon)\Vert u_\gamma\Vert_{L^{\infty}(\Omega\times(0,T))}\Vert \nabla u_\gamma\Vert_{L^2(\Omega\times(0,T))} \Vert \nabla \varphi\Vert_{L^{2}(\Omega\times(0,T))}\\&+
     C(\varepsilon)\Vert u_\gamma\Vert_{L^\infty(\Omega\times (0,T))}\Vert (\nabla \omega_\varepsilon\ast u)_\gamma\Vert_{L^2(\Omega\times(0,T))}\Vert \nabla \varphi\Vert_{L^2(\Omega\times(0,T))}\\&+C\Vert u_\gamma\Vert_{L^{\infty}(\Omega\times(0,T))}\Vert p_H-u^\gamma_\gamma\Vert_{L^2(\Omega\times(0,T))} \Vert  \varphi\Vert_{L^{2}(\Omega\times(0,T))}\\\leq \, & 
     C(\varepsilon,T) \Vert \varphi\Vert_{L^2(0,T;H^1(\Omega))},
     \end{align*}
Therefore, for any $\gamma\geq \overline{\gamma}$, we infer that $\Vert \partial_t u_\gamma\Vert_{L^2(0,T;(H^1(\Omega)^\prime)} \leq C(T)$, thus showing~\eqref{dtu1}.
\\

It remains to prove \eqref{bb}. First note that, clearly, $\partial_t u_\gamma$ and $\partial_t p_\gamma$ share the same sign since $u_\gamma\geq0$ almost everywhere in $\Omega\times(0,T)$. Then we differentiate in time \eqref{nonlocal} and get
\begin{align*}
    \partial_ {tt} u_\gamma-\dfrac {\gamma}{\gamma+1}\Delta \partial_t (u_\gamma^{\gamma+1})-&\frac{1}{2\varepsilon^2}\Delta\partial_t (u_\gamma^2)+\frac{1}{\varepsilon^2}\DIV(\partial_t u_\gamma(\nabla\omega_\varepsilon\ast u_\gamma))+\frac{1}{\varepsilon^2}\text{div}(u_\gamma(\nabla\omega_\varepsilon\ast \partial_t u_\gamma))\\&=\partial_t u_\gamma(p_H-p_\gamma)-u_\gamma\partial_tp_\gamma.
\end{align*}
We test it against $\text{sign}(\partial_t u_\gamma)$ and use Kato's inequality to obtain 
\begin{align*}
    \partial_t \vert \partial_t u_\gamma\vert \leq & \frac{\gamma}{\gamma+1}\Delta(\vert \partial_t(u_\gamma^{\gamma+1})\vert)+\frac{1}{2\varepsilon^2}\Delta(\vert \partial_t (u_\gamma^2)\vert)+\frac{1}{\varepsilon^2}\DIV(\vert \partial_t u_\gamma\vert (-\nabla\omega_\varepsilon\ast u_\gamma))\\&-\frac{1}{\varepsilon^2}\text{div}(u_\gamma(\nabla\omega_\varepsilon\ast \partial_t u_\gamma)) \, \text{sign}(\partial_tu_\gamma)+\vert \partial_tu_\gamma\vert p_H-p_\gamma\vert \partial_t u_\gamma\vert-u_\gamma\vert\partial_t p_\gamma\vert.
\end{align*}
Now we rearrange the terms and integrate in space, deducing
    \begin{align*}
    &\nonumber\frac{\diff}{\diff t}\int_\Omega \vert \partial_t u_\gamma\vert \diff x+\int_\Omega u_\gamma\vert \partial_t p_\gamma\vert \diff x+\int_\Omega p_\gamma\vert \partial_t u_\gamma\vert \diff x\\&\leq -\frac{1}{\varepsilon^2}\int_\Omega\text{div}(u_\gamma(\nabla\omega_\varepsilon\ast \partial_t u_\gamma))\, \text{sign}(\partial_tu_\gamma)\diff x +p_H\int_\Omega \vert \partial_t u_\gamma\vert \diff x.
    \end{align*}
Then we have, by Young's inequality for convolutions,
    \begin{align*}
       -\frac{1}{\varepsilon^2}\int_\Omega & \text{div}(u_\gamma(\nabla\omega_\varepsilon\ast \partial_t u_\gamma)) \, \text{sign}(\partial_tu_\gamma)\diff x\\&=-\frac{1}{\varepsilon^2}\int_\Omega \nabla u_\gamma\cdot (\nabla \omega_\varepsilon\ast \partial_t u_\gamma)\,\text{sign}(\partial_tu_\gamma)\diff x-\frac{1}{\varepsilon^2}\int_\Omega  u_\gamma(\Delta\omega_\varepsilon\ast \partial_t u_\gamma )\,\text{sign}(\partial_tu_\gamma)\diff x\\&\leq \frac{1}{\varepsilon^2}\Vert \nabla u_\gamma\Vert_{L^1(\Omega)}\Vert \nabla \omega_\varepsilon\ast \partial_t u_\gamma\Vert_{L^\infty(\Omega)}+\frac{1}{\varepsilon^2}\Vert u_\gamma\Vert_{L^1(\Omega)}\Vert \Delta\omega_\varepsilon\ast \partial_t u_\gamma\Vert_{L^\infty(\Omega)}\\&\leq C(\varepsilon)(1+ \Vert \nabla u_\gamma\Vert_{L^1(\Omega}))\int_\Omega \vert \partial_t u_\gamma\vert \diff x,
    \end{align*}
since
$$
\Vert \nabla \omega_\varepsilon\ast \partial_t u_\gamma\Vert_{L^\infty(\Omega)}\leq \Vert\nabla \omega_\varepsilon\Vert_{L^\infty(\Omega)}\int_\Omega\vert \partial_tu_\gamma\vert \diff x\leq C(\varepsilon)\int_\Omega\vert \partial_tu_\gamma\vert \diff x,
$$
and the same for $\Vert \Delta\omega_\varepsilon\ast \partial_t u_\gamma\Vert_{L^\infty(\Omega)}$. Therefore we end up with
\begin{align}\label{eq:ddt_time_der_u}
\frac{\diff}{\diff t}\int_\Omega \vert \partial_t u_\gamma\vert \diff x+\int_\Omega u_\gamma\vert \partial_t p_\gamma\vert \diff x+\int_\Omega p_\gamma\vert \partial_t u_\gamma\vert \diff x\leq 
      C(\varepsilon)(1+ \Vert \nabla u_\gamma\Vert_{L^1(\Omega}))\int_\Omega \vert \partial_t u_\gamma\vert \diff x.
    \end{align}
To apply the Gronwall inequality, we need to estimate $\Vert \partial_t u_\gamma(0)\Vert_{L^1(\Omega)}$. From \eqref{nonlocal} we have 
\begin{align*}
\nonumber\Vert \partial_t u_\gamma(0)\Vert_{L^1(\Omega)}\leq\, & C(\varepsilon)\Big(\frac{\gamma}{\gamma+1}\Vert \Delta (u_0^{\gamma+1})\Vert_{L^1(\Omega)}+\Vert \Delta (u_0^2)\Vert_{L^1(\Omega)} +\\&+\Vert \text{div}(u_0(\nabla\omega_\varepsilon\ast u_0)) \Vert_{L^1(\Omega)}+\Vert u_0(p_H-p_0)\Vert_{L^1(\Omega)}\Big).
\end{align*}
Due to Assumption \ref{initial2}, the first term is bounded. Concerning next terms, we have by Assumption \ref{initial1} and \ref{initial2}
$$
\Vert \Delta (u_0^2)\Vert_{L^1(\Omega)}\leq 2\Vert \nabla u_0\Vert_{L^2(\Omega)}^2+2\Vert u_0\Delta u_0\Vert_{L^1(\Omega)}\leq 2\Vert \nabla u_0\Vert_{L^2(\Omega)}^2+2p_H^\frac 1 \gamma \Vert \Delta u_0\Vert_{L^1(\Omega)}\leq C,
$$
\begin{align*}
    \Vert \text{div}(u_0 &(\nabla\omega_\varepsilon\ast u_0)) \Vert_{L^1(\Omega)}\leq \Vert \nabla u_0\cdot (\nabla \omega_\varepsilon\ast u_0)\Vert_{L^1(\Omega)}+\Vert u_0(\Delta \omega_\varepsilon\ast u_0)\Vert_{L^1(\Omega)}
    \\
    &\leq \Vert \nabla u_0\Vert_{L^1(\Omega)}\Vert \nabla \omega_\varepsilon\Vert_{L^\infty(\Omega)}\Vert u_0\Vert_{L^1(\Omega)}+p_H^{\frac 1 \gamma}\vert\Omega\vert\Vert \Delta \omega_\varepsilon\Vert_{L^\infty(\Omega)}\Vert u_0\Vert_{L^1(\Omega)}\leq C(\varepsilon).
\end{align*}
%
It follows that $\nonumber\Vert \partial_t u_\gamma(0)\Vert_{L^1(\Omega)} \leq C(\varepsilon)$. Since $u_\gamma\in L^2(0;T;H^1(\Omega))$ uniformly in $\gamma$ (Lemma \ref{exist}), we may apply the Gronwall Lemma in~\eqref{eq:ddt_time_der_u} and obtain
\begin{align}
\Vert\partial_t u_\gamma\Vert_{L^\infty(0,T;L^1(\Omega))}+\Vert u_\gamma\partial_t p_\gamma\Vert_{L^1(0,T;L^1(\Omega))}\leq C(T,\varepsilon),
\label{a1}
\end{align}
with $C=C(T,\varepsilon)>0$ independent of $\gamma$. To conclude the argument we notice that 
\begin{align*}
    \int_\Omega \vert \partial_t p_\gamma\vert \diff x\leq \int_{u_\gamma\leq \frac{1}{2}}\vert \partial_t p_\gamma\vert \diff x+ 2\int_{u_\gamma>\frac1 2 }u_\gamma\vert \partial_t p_\gamma\vert \diff x\leq \frac{\gamma}{2^{\gamma-1}}\int_\Omega\vert \partial_tu_\gamma \vert \diff x+ 2\int_{\Omega}u_\gamma\vert \partial_t p_\gamma\vert \diff x,
\end{align*}
so that, being $\frac{\gamma}{2^{\gamma-1}}\leq 1$ for any $\gamma\geq1$, from \eqref{a1} we deduce the second estimate in~\eqref{bb} 
and conclude the proof of Lemma~\ref{ests}.
\end{proof}

\subsection{The limit $\gamma\to \infty$}
We complete the convergence as $\gamma\to \infty$ distinguishing two steps.
\paragraph{Step 1. Consequences of Lemmas \ref{exist}, \ref{DeG} and \ref{ests}.}
 Exploiting those lemmas, we can obtain the following convergences, up to subsequences, which are deduced by standard compactness arguments: for any $T>0$, as $\gamma\to \infty$,
\begin{align}
\label{convb}
&u_\gamma\to u_\infty \text{ almost everywhere in }\Omega\times(0,T),\\&
u_\gamma\overset{\ast} {\rightharpoonup}u_\infty\text{ in }L^\infty(\Omega\times(0,T)),\\&
u_\gamma\to u_\infty \text{ in }L^p(\Omega\times(0,T))\quad \forall p\in[1,+\infty),\label{pp}\\&
u_\gamma\rightharpoonup u_\infty \text{ in }L^2(0,T;H^1(\Omega)),\\&
\partial_t u_\gamma\rightharpoonup \partial_t u_\infty \text{ in }L^2(0,T;(H^1(\Omega))'),
\\&
p_\gamma\rightharpoonup p_\infty \text{ in }L^2(0,T;H^1(\Omega)), \;  L^3((0,T)\times \Omega).
\label{lastconv}
\end{align}
Moreover, by the Aubin-Lions-Simon Lemma,
\begin{align}
p_\gamma\to p_\infty \text{ in }L^2(0,T;L^2(\Omega)),
    \label{pstrong}
\end{align}
which, thanks to \eqref{est_p}  can be improved by interpolation, to
\begin{align}
p_\gamma\to p_\infty \text{ in }L^q(0,T;L^q(\Omega)),\quad \forall q\in[2,3).
\label{L3conv}
\end{align}
In order to obtain the complementarity condition, we study the function $v_\gamma:=u_\gamma^{\gamma+1}=u_\gamma p_\gamma$. First we have 
$$
\Vert v_\gamma\Vert _{L^2(\Omega\times(0,T))}\leq \Vert u_\gamma\Vert_{L^\infty(\Omega\times(0,T))}\Vert p_\gamma\Vert_{L^2(\Omega\times(0,T))}\leq C(T,\varepsilon),
$$
and
$$
\Vert \nabla v_\gamma\Vert _{L^2(\Omega\times(0,T))}=\frac{\gamma+1}{\gamma}\Vert u_\gamma\nabla p_\gamma\Vert _{L^2(\Omega\times(0,T))}\leq \frac{\gamma+1}{\gamma}\Vert u_\gamma\Vert_{L^\infty(\Omega\times(0,T))}\Vert \nabla p_\gamma\Vert_{L^2(\Omega\times(0,T))}\leq C(T,\varepsilon),
$$
since $\frac{\gamma+1}{\gamma}\to 1$ as $\gamma\to \infty$.
Therefore,  up to subsequences, we have 
$$
v_\gamma=u_\gamma p_\gamma\rightharpoonup v_\infty \qquad \text{in } \; L^2(0,T;H^1(\Omega)),
$$
for some $v_\infty\in L^2(0,T;H^1(\Omega))$. To identify $v_\infty$, we observe that
\begin{align*}
\Vert u_\gamma &p_\gamma-u_\infty p_\infty\Vert_{L^2(\Omega\times (0,T))}\leq \Vert u_\gamma(p_\gamma-p_\infty)\Vert_{L^2(\Omega\times (0,T))}+\Vert p_\infty(u_\gamma-u_\infty) \Vert_{L^2(\Omega\times (0,T))}\\& 
\leq \Vert u_\gamma\Vert_{L^\infty(\Omega\times (0,T))}\Vert p_\gamma-p_\infty\Vert_{L^2(\Omega\times (0,T))}+\Vert u_\gamma-u_\infty\Vert_{L^6(\Omega\times (0,T))}\Vert p_\infty\Vert_{L^3(\Omega\times (0,T))}\to 0
\end{align*}
as $\gamma\to \infty$, thanks to the above results, in particular \eqref{Linfty}, \eqref{est_p}, \eqref{pp} and \eqref{pstrong}.
From this we clearly identify $v_\infty=u_\infty p_\infty$ and obtain 
\begin{align}
&v_\gamma=u_\gamma p_\gamma\rightharpoonup v_\infty=u_\infty p_\infty \text{ in } L^2(0,T;H^1(\Omega)).
\label{ppp}\\&
v_\gamma\to p_\infty u_\infty\text{ in }L^2(\Omega\times(0,T)). \label{eq:strong_conv_v_gamma} 
\end{align}
We are now able to pass to the limit in $\gamma$ to obtain \eqref{incom} and 
\begin{align}
p_\infty(1-u_\infty)=0,\quad 0\leq u_\infty\leq 1 ,\quad p_\infty\geq 0.
\label{compa}
\end{align}
Indeed, we can argue as in \cite{MR1687274} to see that, for any $\epsilon>0$ there exists $\gamma_0=\gamma_0(\epsilon)$ such that 
for any $y\geq0$ and $\gamma\geq \gamma_0$
$$
y^{\gamma+1}\geq y^\gamma-\epsilon.
$$
Applying this to $y=u_\gamma$, since we have that, up to subsequences, $u^\gamma=p_\gamma\to p_\infty$ and $v_\gamma=u_\gamma^{\gamma+1}\to p_\infty u_\infty$ almost everywhere in $\Omega\times(0,T)$, we can pass to the limit and obtain, 
$$
u_\infty p_\infty\geq p_\infty-\epsilon, \qquad \forall \epsilon. >0, 
$$
which implies that $p_\infty u_\infty\geq p_\infty$. Since by \eqref{Linfty} and \eqref{convb} we get $0\leq u_\infty\leq 1$ almost everywhere in $\Omega\times(0,T)$, we get $p_\infty\geq p_\infty u_\infty\geq p_\infty$, i.e., \eqref{compa}, since it holds for any $T>0$. In order to pass to Step 2, we introduce the quantities 
\begin{equation}\label{eq:v_gamma_inf_tilde}
\widetilde{v}_\gamma:= \frac{\gamma}{\gamma+1}v_\gamma+\frac 1 {2\varepsilon^2}u_\gamma^2, \qquad  \widetilde{v}_\infty = v_\infty + \frac 1 {2\varepsilon^2}u_\infty^2,
\end{equation}
 and study $\nabla \widetilde{v}_\gamma$ which is essential to obtain the complementarity condition~\eqref{eq:incompressible_pressure}.
\paragraph{Step 2. Strong convergence of $\nabla \widetilde{v}_\gamma$.}
\begin{lemma}
Let $\widetilde{v}_\gamma$, $\widetilde{v}_\infty$ be as in \eqref{eq:v_gamma_inf_tilde}. Then, for any $T>0$,
\label{compactness}
$$
\widetilde{v}_\gamma \to \widetilde{v}_\infty \quad \text{in }L^2(0,T;H^1(\Omega))\quad\text{ as }\gamma\to \infty.
$$
\end{lemma}
\begin{proof}
Using \eqref{pp} and \eqref{eq:strong_conv_v_gamma} we obtain $\widetilde{v}_\gamma \to \widetilde{v}_\infty$ in $L^2(\Omega\times (0,T))$ so it is sufficient to prove $\nabla \widetilde{v}_\gamma \to \nabla \widetilde{v}_\infty$ in $L^2(\Omega\times (0,T))$. Of course, by \eqref{convb}-\eqref{lastconv} and \eqref{ppp}, we have weak convergence 
\begin{align}
\label{weak}
\nabla \widetilde{v}_\gamma \rightharpoonup \nabla \widetilde{v}_\infty \,  \text{in }L^2(0,T;L^2(\Omega)),\quad\gamma\to \infty.
\end{align} Let us first observe that \eqref{nonlocal} can be rewritten highlighting the presence of $\widetilde{v}_\gamma$: 
\begin{align}
\partial_t u_\gamma	 - \Delta \widetilde{v}_\gamma =u_\gamma	G(p_\gamma	)-\frac 1 {\varepsilon^2}\text{div}(u_\gamma	(\nabla \omega_\varepsilon\ast u_\gamma	)).
\label{eq}
\end{align}
We multiply \eqref{eq} by $\widetilde{v}_\gamma	-\widetilde{v}_\infty$ and integrate over $\Omega_T:=\Omega\times(0,T)$. Since
$$
\int_\Omega\partial_t u_\gamma	 v_\gamma	\diff x=\frac{1}{\gamma+2}\frac{\diff}{\diff t}\int_\Omega u_\gamma	^{\gamma+2}\diff x.
$$
we obtain
\begin{align}
& \nonumber\frac{\gamma}{(\gamma+2)(\gamma+1)}\int_\Omega u_\gamma	^{\gamma+2}(T)\diff x +\int_{\Omega_T} \nabla \widetilde{v}_\gamma	\cdot \nabla(\widetilde{v}_\gamma	-\widetilde{v}_\infty)\diff x \diff s\\&=\nonumber\frac{\gamma}{(\gamma+2)(\gamma+1)}\int_\Omega u_\gamma	^{\gamma+2}(0)\diff x-\frac 1 {2\varepsilon^2} \int_{\Omega_T}\partial_t u_\gamma	( u_\gamma	^2-u_\infty^2)\diff x\diff t+\int_{\Omega_T} u_\gamma	G(p_\gamma	)(\widetilde{v}_\gamma	-\widetilde{v}_\infty)\diff x \diff t\\&+\frac 1 {\varepsilon^2}\int_{\Omega_T}  u_\gamma	(\nabla\omega_\varepsilon\ast u_\gamma)\cdot \nabla(\widetilde{v}_\gamma	-\widetilde{v}_\infty)\diff x \diff t+\int_0^{T}\langle\partial_tu_\gamma	,v_\infty\rangle \diff t.
\label{p11}
\end{align}
The plan is to estimate $\limsup_{\gamma \to \infty }\int_{\Omega_T}|\nabla \widetilde{v}_\gamma-\nabla \widetilde{v}_\infty|^2 \diff x \diff t$ from \eqref{p11}. First, we rewrite the term $\frac 1 {2\varepsilon^2} \int_{\Omega_T}\partial_t u_\gamma	( u_\gamma	^2-u_\infty^2)\diff x \diff t$. We have 
\begin{align*}
&\frac 1 {2\varepsilon^2}\int_{\Omega}\partial_t u_\gamma	( u^2_\gamma	-u_\infty^2)\diff x=
\frac 1 {2\varepsilon^2} \langle\partial_t u_\gamma	-\partial_t u_\infty, u_\gamma	^2-u_\infty^2\rangle 
+\frac 1 {2\varepsilon^2} \langle\partial_t u_\infty,u^2_\gamma	-u_\infty^2\rangle \\& 
=\frac 1  {4\varepsilon^2}\frac \diff {\diff t}\int_{\Omega}(u_\gamma	- u_\infty)^2(u_\gamma	+u_\infty)\diff x-\frac 1 {4\varepsilon^2}\langle\partial_t(u_\gamma	+u_\infty),(u_\gamma	-u_\infty)^2\rangle +\frac 1 {2\varepsilon^2}\langle\partial_t u_\infty, u_\gamma	^2-u_\infty^2 \rangle
\\&
=\frac 1 {4\varepsilon^2}\frac \diff {\diff t}\int_{\Omega}(u_\gamma	- u_\infty)^2(u_\gamma	+u_\infty)\diff x-\frac 1 {12\varepsilon^2}\frac \diff {\diff t}\int_{\Omega}(u_\gamma	-u_\infty)^3\diff x -\frac 1 {2\varepsilon^2}\langle\partial_t u_\infty, (u_\gamma	-u_\infty)^2 \rangle \\&
\phantom{ = } +\frac 1 {2\varepsilon^2} \langle\partial_t u_\infty, u^2_\gamma	-u_\infty^2 \rangle,
\end{align*}
so that, integrating over $[0,T]$ and recalling that $u_\infty(0)\equiv u_\gamma(0)\equiv u_0$, we get 
\begin{align}\label{eq:expression_for_derivative_test_vgammatilde}
&\nonumber -\frac 1 {2\varepsilon^2}\int_{\Omega_T}\partial_t u_\gamma	( u_\gamma	^2-u_\infty^2)\diff x \diff t =-\frac 1 {4\varepsilon^2} \int_{\Omega}(u_\gamma	(T)- u_\infty(T))^2(u_\gamma	(T)+u_\infty(T))\diff x\\&+\frac{1}{12\varepsilon^2}\int_{\Omega}(u_\gamma	(T)-u_\infty(T))^3\diff x
+\frac 1 {2\varepsilon^2}\int_0^{T}\langle \partial_t u_\infty, (u_\gamma	-u_\infty)^2\rangle \diff t -\frac 1 {2\varepsilon^2} \int_0^{T}\langle \partial_t u_\infty, u^2_\gamma	-u_\infty^2\rangle \diff t.
\end{align}
Note that 
\begin{equation}\label{eq:sign_obs_exp_derivative}
\begin{split}
\frac 1 {4\varepsilon^2} \int_{\Omega} & (u_\gamma	(T)- u_\infty(T))^2(u_\gamma	(T)+u_\infty(T))\diff x-\frac 1 {12\varepsilon^2}\int_{\Omega}(u_\gamma	(T)-u_\infty(T))^3\diff x\\&=\frac 1 {4\varepsilon^2} \int_\Omega(u_\gamma	(T)- u_\infty(T))^2\left(u_\gamma	(T)+u_\infty(T)-\frac 1 3 u_\gamma	(T)+\frac 1 3 u_\infty(T) \right)\diff x \geq 0.
\end{split}
\end{equation}
Therefore, taking into account \eqref{eq:expression_for_derivative_test_vgammatilde}, \eqref{eq:sign_obs_exp_derivative} and $u_{\gamma} \geq 0$ we obtain from \eqref{p11} 
\begin{align}
\nonumber\int_{\Omega_T} \nabla \widetilde{v}_\gamma	\cdot & \nabla(\widetilde{v}-\widetilde{v}_\infty)\diff x \diff s \leq \frac{\gamma}{(\gamma+2)(\gamma+1)}\int_\Omega u_\gamma	^{\gamma+2}(0)\diff x+ \frac 1 {2\varepsilon^2}\int_0^{T}\langle\partial_t u_\infty, (u_\gamma	-u_\infty)^2\rangle \diff t\\&-\int_0^{T}\langle\partial_t u_\infty,\frac 1 {2\varepsilon^2} u_\gamma	^2-\frac 1 {2\varepsilon^2}u_\infty^2 \rangle \diff t+\int_{\Omega_T} u_\gamma	G(p_\gamma)(\widetilde{v}_\gamma	-\widetilde{v}_\infty)\diff x \diff t\nonumber
\\
&+\frac 1 {\varepsilon^2}\int_{\Omega_T}  u_\gamma	(\nabla\omega_\varepsilon\ast u_\gamma	)\cdot \nabla(\widetilde{v}_\gamma	-\widetilde{v}_\infty)\diff x \diff t+\int_0^{T}\langle\partial_tu_\gamma	,v_\infty \rangle \diff t.
\label{p13}
\end{align}
Observe that 
\begin{align*}
\int_{\Omega_T} \nabla \widetilde{v}_\gamma	\cdot \nabla(\widetilde{v}_\gamma	-\widetilde{v}_\infty)\diff x \diff s=\int_{\Omega_T} \vert\nabla(\widetilde{v}_\gamma	-\widetilde{v}_\infty)\vert^2\diff x \diff s+\int_{\Omega_T} \nabla \widetilde{v}_\infty\cdot \nabla(\widetilde{v}_\gamma	-\widetilde{v}_\infty)\diff x \diff s,
\end{align*}
but by the weak convergence in \eqref{weak} we have $$
\int_{\Omega_T} \nabla \widetilde{v}_\infty\cdot \nabla(\widetilde{v}_\gamma	-\widetilde{v}_\infty)\diff x \diff s\to 0,
$$
since $\nabla\widetilde{v}_\infty\in L^2(0,T;L^2(\Omega))$. Thus we deduce that 
\begin{align*}
\limsup_{\gamma\to \infty}\int_{\Omega_T} \vert\nabla(\widetilde{v}_\gamma	-\widetilde{v}_\infty)\vert^2\diff x \diff s=\limsup_{\gamma\to \infty}\int_{\Omega_T} \nabla \widetilde{v}_\gamma	\cdot \nabla(\widetilde{v}_\gamma	-\widetilde{v}_\infty)\diff x \diff s.
\end{align*}
Then, recalling also that $\partial_t u_\gamma	 \rightharpoonup \partial_t u_\infty$ in $L^2(0,T;(H^1(\Omega))')$ and that $v_\infty\in L^2(0,T;H^1(\Omega))$, we deduce from \eqref{p13}
\begin{align}
&\nonumber\limsup_{\gamma\to \infty}\int_{\Omega_T} \vert\nabla(\widetilde{v}_\gamma-\widetilde{v}_\infty)\vert^2\diff x \diff s \leq \limsup_{\gamma\to \infty}\left(\frac{\gamma}{(\gamma+2)(\gamma+1)}\int_\Omega u_\gamma^{\gamma+2}(0)\diff x\right. \\
&\nonumber+\frac 1 {2\varepsilon^2}\int_0^{T}\langle\partial_t u_\infty, (u_{\gamma}-u_\infty)^2\rangle \diff t \left.-\frac 1 {2\varepsilon^2} \int_0^{T}\langle\partial_t u_\infty,  u_\gamma^2-u_\infty^2 \rangle \diff t+\int_{\Omega_T} u_\gamma G(p_\gamma)(\widetilde{v}_\gamma-\widetilde{v}_\infty)\diff x \diff t\right.\\&\left.+\frac 1 {\varepsilon^2}\int_{\Omega_T}  u_\gamma\nabla\omega_\varepsilon\ast u_\gamma\cdot\nabla(\widetilde{v}_\gamma-\widetilde{v}_\infty)\diff x \diff t\right)+\int_0^{T}\langle\partial_tu_\infty,v_\infty \rangle \diff t.
\label{ppp1}
\end{align}
The plan is to prove that all the terms on the (RHS) of \eqref{ppp1} converge to 0. First, we have that $0\leq u_\gamma(0)=u_0\leq p_H^\frac 1 \gamma$ for any $\gamma>0$ by assumption, so that in the end we deduce $u_\gamma^{\gamma+1}(0)\leq p_H^{\frac{\gamma+2}{\gamma}}\leq C$ and thus, as $\gamma \to \infty$,
$$
\frac{\gamma}{(\gamma+2)(\gamma+1)}\int_\Omega u_\gamma^{\gamma+2}(0)\diff x\to 0.
$$
Then, using \eqref{pp} and the fact that $\partial_t u_\infty\in L^2(0,T;(H^1(\Omega))^\prime)$, we immediately deduce 
$$
\frac 1 {2\varepsilon^2}\int_0^{T}\langle\partial_t u_\infty, (u_\gamma-u_\infty)^2 \rangle \diff t-\frac 1 {2\varepsilon^2} \int_0^{T}\langle\partial_t u_\infty,  u_\gamma^2-u_\infty^2 \rangle \diff t\to 0\quad \text{ as }\gamma\to \infty.
$$
Concerning the term $\int_{\Omega_T} u_\gamma G(p_\gamma)(\widetilde{v}_\gamma-\widetilde{v}_\infty)\diff x \diff t$, we simply use the fact that $\{u_{\gamma}\}$ is bounded in $L^{\infty}(\Omega\times (0,T))$ (cf. \eqref{Linfty}), $\{G(p_{\gamma})\}$ is bounded in $L^2(\Omega\times (0,T))$ (cf. \eqref{est_p}) and $\widetilde{v}_\gamma \to \widetilde{v}_\infty$ strongly in $L^2(\Omega\times (0,T))$. Similarly, $\int_{\Omega_T} u_\gamma\nabla\omega_\varepsilon\ast u_\gamma\cdot\nabla(\widetilde{v}_\gamma-\widetilde{v}_\infty) \diff x \diff t \to 0$ because of weak convergence \eqref{weak} and strong convergence $u_\gamma\nabla\omega_\varepsilon\ast u_\gamma \to u_{\infty} \nabla\omega_\varepsilon\ast u_\infty$ in $L^2(\Omega\times (0,T))$ (which follows by \eqref{pp} and simple properties of convolutions).\\

We are left with the analysis of the last term, i.e., $\int_0^{T}\langle\partial_t u_\infty, v_\infty \rangle \diff t$. Our aim is to show that this term vanishes, exploiting Theorem~\ref{Har}. We introduce the following indicator function on $\R$:
\begin{align}
I_S(s):=\begin{cases}
0\quad \text{ if }s\leq 1,\\
+\infty\text{ if }s>1,
\end{cases}
\end{align}
 and define $S=(-\infty,1]$, which is a closed, convex and nonempty set, so that $I_S: \R\to (-\infty,+\infty]$ is proper, convex and lower semicontinuous (see, e.g., \cite[Appendix 1]{MR2183776}) and it holds
$$
\partial I_S(x)=\{y\in \R: y\cdot(x-s)\geq0,\quad \forall s\leq 1\}.
$$
We see that we have

\begin{itemize}
	\item $u_\infty\in L^2(0,T;H^1(\Omega))$ and $\partial_t u_\infty\in L^2(0,T; (H^1(\Omega))')$;
	\item  for almost any $(x,t)\in\Omega_T$, 
	$$
	v_\infty(x,t)\in \partial I_S(u_\infty(x,t)),
	$$
	by \eqref{compa} and being $v_\infty=p_\infty$. Indeed,
	$$
	v_\infty(x,t)(u_\infty(x,t)-s)\geq 0,\quad \forall s\leq 1,
	$$
	since, when $u_\infty(x,t)=1$, being $v_\infty(x,t)\geq0$, the inequality is always verified for any $s\leq1$, whereas, when $u_\infty(x,t)<1$, it holds $v_\infty(x,t)=0$ and thus the inequality is verified for any $s\leq1$ as well.
	\item $v_\infty\in L^2(0,T;H^1(\Omega))$.
\end{itemize}
Therefore, all the assumptions are verified and we can apply Theorem \ref{Har} with $h=I_S$, $f=u_\infty$, $g=v_\infty$, to infer, after an integration over $[0,T]$, 
\begin{align}
    0=\int_\Omega I_S(u_\infty(x,T))\diff x-\int_\Omega I_S(u_\infty(x,0))\diff x=\int_0^{T}\langle\partial_t u_\infty,v_\infty \rangle \diff t,
    \label{AC}
\end{align}
by the definition of $I_S$. 
Indeed, it holds $\int_\Omega I_S(u_\infty(x,\cdot))\diff x\equiv 0$ on $[0,T]$ for any $T>0$. To see this, first notice that, being $u_\infty\leq 1$ almost everywhere in $\Omega\times[0,\infty)$, we deduce that, for almost any $t\in[0,\infty)$,
$$
u_\infty(x,t)\leq 1\quad\text{ for almost any }x\in\Omega.
$$ 
Therefore, $\int_\Omega I_S(u_\infty(x,\cdot))\diff x=0$, for almost any $t\in[0,T]$ and for any $T>0$. Recall now by Theorem \ref{Har} that $\int_\Omega I_S(u_\infty(x,\cdot))\diff x\in AC([0,T])$ for any $T>0$, which ensures that $\int_\Omega I_S(u_\infty(x,t))\diff x\equiv 0$ for \textit{any} $t\in[0,T]$ and for any $T>0$.\\ 

Having studied all the terms in the right-hand side of \eqref{ppp1}, in the end we conclude that 
$$
\limsup_{\gamma\to \infty}\int_{\Omega_T} \vert\nabla(\widetilde{v}_\gamma-\widetilde{v}_\infty)\vert^2\diff x \diff s\leq0,
$$
implying that 
$$
\lim_{\gamma\to \infty}\int_{\Omega_T} \vert\nabla(\widetilde{v}_\gamma-\widetilde{v}_\infty)\vert^2\diff x \diff s=0.
$$
\end{proof} 
\paragraph{Step 3. The complementarity condition.}
Now that we have all the necessary convergences, let us consider the following equation in distributional sense (this equation comes from multiplying \eqref{nonlocal} by $v_\gamma=u_\gamma^{\gamma+1}$):
\begin{align*}
\frac{1}{\gamma+2}\partial_t u^{\gamma+2}_\gamma-v_\gamma\Delta \widetilde{v}_\gamma+\frac 1 {\varepsilon^2}v_\gamma\text{div}(u_\gamma(\nabla\omega_\varepsilon\ast u_\gamma))=u_\gamma G(p_\gamma)v_\gamma.
\end{align*}
Thanks to the results of Step 2. and Lemma \ref{compactness}, we can then pass the limit as $\gamma\to \infty$ and obtain the complementarity condition:
\begin{align}
\label{compat}
v_\infty(\Delta \widetilde{v}_\infty-\frac 1 {\varepsilon^2}\text{div}(u_\infty(\nabla\omega\ast u_\infty))+u_\infty G(p_\infty))=0\quad \text{ in }\mathcal{D}'(\Omega\times(0,T)),
\end{align}
for any $T>0$, so that in the end, recalling $v_\infty=p_\infty$ and $\widetilde{v}_\infty=p_\infty+\frac 1 {2\varepsilon^2}u_\infty^2$,
\begin{align}
p_\infty(\Delta {p}_\infty+\frac 1 {2\varepsilon^2} \Delta u_\infty^2-\frac 1 {\varepsilon^2}\text{div}(u_\infty(\nabla\omega_\varepsilon\ast u_\infty))+u_\infty G(p_\infty))=0\quad \text{ in }\mathcal{D}'(\Omega\times(0,\infty)),
\end{align}
which is the complementarity condition \eqref{eq:incompressible_pressure}. In conclusion, since $v_\infty=p_\infty$, we can repeat the argument leading to \eqref{AC}, to infer that, for any $t\in[0,\infty)$,
$$
\int_0^t \langle\partial_t u_\infty, p_\infty \rangle \diff s=0, \qquad \langle\partial_t u_\infty(t), p_\infty(t) \rangle =0,
$$
thus concluding the proof of Theorem~\ref{thm:incompressible}.

\section{Convergence to equilibria: proof of Theorem~\ref{thm:conv_stationary}}
\label{section::equil}



Here, we prove Theorem \ref{thm:conv_stationary}. Numerical simulations illustrating the result in dimension $1$ with a source term are depicted in Figure~\ref{fig:asymptotics_source}. 


\subsection{Case $G(p)=p_{H}-p$}

\textit{Proof of Theorem \ref{thm:conv_stationary}.} We now divide the proof of Theorem \ref{thm:conv_stationary} in different steps.
\\

\textbf{Step 1: characterization of possible limits.} We fix the initial datum $u_0 \ne 0$ as in the statement of Theorem~\ref{thm:conv_stationary}. 
We show the following key result: 

 \begin{lemma}
From any divergent sequence $\{t_n\}_n\subset \R^+$ we can extract a subsequence (not relabeled) such that,  as $n\to\infty$,  $u_{n}(t):= u(t+t_{n})$ converges to the same limit $ u_{\ast}$ strongly in $L^{p}_{t,x}$
for all $1\le p<2\gamma+1$. Moreover, either  $u_{\ast}\equiv0$ and $\Phi(u(t))\to |\Omega|$ as $t\to\infty$, otherwise  $u_{\ast}\equiv p_H^{1/\gamma}$ and $\Phi(u(t))\to 0$ as $t\to\infty$.

\label{ll}  \end{lemma}


We fix $T>0$, we also consider $n$ large enough such that $t_n>T$. Observe that $u_n$ solves the problem
\begin{equation} \label{nonlocalA}
		\begin{cases}
		\partial_t u_n -\text{div}(u_n\nabla\mu_n)=u_nG(p_n)\quad \text{in }\Omega\times  (-T, T), 
  \\
		\mu_n=p_n+B_\varepsilon(u_n),\qquad  \quad p_n=u_n^\gamma.
		\end{cases}
\end{equation}
By Propositions \ref{prop:energy/entropy} and \ref{prop:unif_energy_time}, we have the following uniform-in-$n$ bounds :
\begin{align}
		\label{PhiA}&
		\mathcal{E}(u_n(t))\leq C	\quad \forall t\in[-T, T],\\&
		\Phi(u_n(t))\leq \Phi(u_0) \quad \forall t\in (-T, T), \qquad \Phi(u_n(t)) \text{ non-increasing}.
\label{PhiB}\end{align}
		From this and  Propositions~\ref{prop:energy/entropy} and~\ref{partialtu}, we deduce the following uniform bounds, for any $T>0$,  
\begin{equation}		\label{unif}
\begin{split}
		&\Vert u_n\Vert_{L^\infty(-T,T;L^{\gamma+1}(\Omega))}+\Vert u_n\Vert_{L^2(-T,T;H^1(\Omega))}
		+\Vert u_n \Vert_{L^{2\gamma+1}(\Omega\times(-T,T))}\\&+\Vert \partial_t u_n\Vert_{L^{q^\prime}(-T,T;(W^{1,q}(\Omega))')}+\Vert \partial_t u_n\Vert_{L^{2}(-T,T;(W^{1,r}(\Omega))')}+\Vert \nabla u_n^\frac{\gamma+1}{2}\Vert_{L^2(-T,T;L^2(\Omega))}\leq C(T),
\end{split}
\end{equation}
with 
$q$ and $r$ as in Proposition~\ref{partialtu} which implies, by standard arguments, the following convergences (up to subsequences) as $n\to\infty$ to the same function $u_{\ast} \geq 0$ 
\begin{align*}
  &	u_n\rightharpoonup u_{\ast} \quad \text{in }L^2(-T,T;H^1(\Omega))  \quad \text{and }  L^{2\gamma+1}(\Omega\times(-T,T))
\\
  &	\partial_t u_n\rightharpoonup \partial_t u_{\ast} \quad \text{in }L^{q^\prime}(-T,T;(W^{1,q}(\Omega))')
    \quad \text{and in }L^{2}(-T,T;(W^{1,r}(\Omega))'),
\\
 &u_n\to u_{\ast} \quad \text{in }L^p(\Omega\times(0,T)), \quad p\in [1, 2\gamma+1),\qquad  \text{and almost everywhere},
\\
 & \nabla u_n^\frac{\gamma+1}{2}\rightharpoonup \nabla u_{\ast}^\frac{\gamma+1}{2}  \quad \text{in }L^2(\Omega\times (0,T)).
\end{align*}
We want to characterize $u_{\ast}$. Notice that from \eqref{ent}, we have for all $T>0$,
\begin{align*} 
	  \int_{-T}^T \int_\Omega\int_\Omega &\omega_\varepsilon(y)\vert \nabla u_n(x)-\nabla u_n(x-y) \vert^2\diff x \diff y \diff s
\\
&  = \int_{t_n-T}^{t_n+T} \int_\Omega\int_\Omega \omega_\varepsilon(y)\vert \nabla u(x)-\nabla u(x-y) \vert^2\diff x \diff y \diff s   \to 0  \qquad \text{as } n \to \infty
\end{align*}
by integrability on $(0, \infty)$. By weak-lower semicontinuity, in the limit 
\[
\int_{-T}^T \int_\Omega\int_\Omega \omega_\varepsilon(y)\vert \nabla u_{\ast}(x)-\nabla u_{\ast}(x-y) \vert^2\diff x \diff y \diff s = 0
\]
so that $u_{\ast}(t)$ is constant in space for a.e. $t \in [0,T]$. In fact, $u_{\ast}(t)$ is constant in space \text{for all} $t\in[0,T]$ because $u_{\ast}\in C_{weak}([-T,T];L^{\gamma+1}(\Omega))$  (see Remark~\ref{continu}) so that for all $\varphi \in C_c^{\infty}(\Omega)$, the function $t \mapsto \int_{\Omega} u_{\ast}(t,x) \, \DIV \varphi(x) \diff x$ is continuous. Similarly, from~\eqref{ent} and the Fatou lemma, 
$$
u_\ast\log\left(\frac{u_{\ast}}{p_H^{\frac 1 \gamma}}\right)(u_{\ast}^{\gamma} - p_H) = 0,  \text{ for a.e. } t \in (-T,T),
$$ 
so that either $u_{\ast}(t) = 0$ or $u_{\ast}(t) = p_H^{\frac{1}{\gamma}}$. Since $u_{\ast}\in C_{weak}([-T,T];L^{\gamma+1}(\Omega))$, the average $\overline{u}_{\ast}(\cdot)\in C([0,T])$ so that $u_\ast\equiv \overline{u}_{\ast}$ can attain only one of the values for all times. 
\\

Because $\Phi(u(t))$ is non-increasing in time, it has a limit as $t\to\infty$, say $\Phi_\ast$. Then clearly $\Phi_\ast=\lim_{n\to\infty}\Phi(u(t+t_n))=\Phi(u_\ast)$, and thus either $\Phi_\ast=\Phi(0)= |\Omega|$ if $u_\ast=0$ or $\Phi_\ast=\Phi(p_H^{\frac 1 \gamma})=0$ if $u_\ast=p_H^{\frac 1 \gamma}$. Clearly this also implies that, given another sequence of times $\{t_m\}_m$, we can repeat the same argument and extract a (non relabeled) subsequence $\{u(t_m)\}_m$ converging to the same constant $u_\ast$. This concludes the proof of Lemma~\ref{ll}.\\

 \textbf{Step 2. Stability of the equilibria.} We complete Lemma~\ref{ll} with the following 
	\begin{lemma}
 \label{lemma:conv}
		Under notation of Lemma \ref{ll}, if $u_{*}\equiv 0$ then $u_0=0$ almost everywhere in $\Omega$. 
	\end{lemma}
	\begin{proof}
		Since $\Phi(u(t))$ is nonincreasing in time and, by Lemma \ref{ll}, $\Phi(u(t))\to \vert \Omega\vert$, we obtain
		\begin{align*}
		|\Omega|= \lim_{t\to \infty} \Phi(u(t ))\leq \Phi(u_0)\leq |\Omega|.
		\end{align*}
  Thus we infer  $\Phi(u_0)= |\Omega|$.
  Being the entropy function  $g:x\mapsto \frac{x}{p_H^{\frac 1 \gamma}}\log\left(\frac{x}{p_H^{\frac 1 \gamma}}\right)-\frac{x}{p_H^{\frac 1 \gamma}}+1$ decreasing for $x \in [0, p_H^{\frac 1 \gamma})$, and since  $u_0 \leq p_H^{\frac 1 \gamma}$, it follows that $u_0=0$ almost everywhere in $\Omega$.
	\end{proof}

\textbf{Step 3. Existence of the $L^{q}(\Omega)$-limit as $t\to\infty$.}

\begin{lemma} \label{convL1} Assume that $u_0\not\equiv 0$. Then it holds
$$  \lim_{t\to\infty} u(t) =p_H^\frac 1 \gamma \quad\text{ in }L^q(\Omega) \qquad \forall q\in[1,\gamma+1). $$
\end{lemma}
\begin{proof} First note that, the pointwise values $u(t)$ as an $L^{\gamma+1}(\Omega)$-function makes sense thanks to the weak continuity obtained in Lemma~\ref{exist}. Now, we consider the decomposition of~$\Phi$:
 \begin{align}\Phi(u)=\dfrac{1}{p_H^\frac 1 \gamma}\left(\int_\Omega u\log\left(\dfrac{u}{\overline{u}}\right)\diff x +\int_\Omega u\log\left(\dfrac{\overline{u}}{p_H^\frac 1 \gamma}\right)\diff x +\int_\Omega (p_H^\frac 1 \gamma-u)\diff x \right)  \label{decomposition}
\end{align}
and we study the limits of terms appearing in \eqref{decomposition}. Since $u_0\not\equiv 0$, from Lemma \ref{ll} we obtain
 $\lim_{t \to \infty} \Phi(u(t)) = \Phi(p_H^{1/\gamma}) = 0$.\\

By Proposition \ref{prop:mass_conserv}, for any sequence $\{t_n\}_n$ there exists a (nonrelabeled) subsequence such that, for some $\kappa$,
\begin{equation}\label{eq:convergence_of_masses}
\overline{u}(t_n)\to \kappa \in [0, p_H^{\frac{1}{\gamma}}].
\end{equation}
We prove that $\kappa = p_H^{\frac{1}{\gamma}}$. Indeed, the function $g:x\mapsto \frac{x}{p_H^{\frac 1 \gamma}}\log\left(\frac{x}{p_H^{\frac 1 \gamma}}\right)-\frac{x}{p_H^{\frac 1 \gamma}}+1$ is convex and continuous. As $\Phi(u(t))=\int_\Omega g(u(t))\diff x $, by Jensen's inequality we get 
$$ 0\leq g(\overline{u}(t))\leq \frac{1}{\vert\Omega\vert}\Phi(u(t))\to 0\text{ as }t\to\infty
$$
so that $g(\kappa)=0$ and the claim follows. It follows that $\overline{u}(t)\to  p_H^{\frac{1}{\gamma}}$ as $t \to \infty$.\\ 

Hence, passing to the limit in \eqref{decomposition} 
\begin{align}  \int_\Omega u(t)\log\left(\dfrac{u(t)}{\overline{u}(t)}\right)\diff x \to 0\text{ when }k\to\infty.\label{tinfty}  \end{align}
From Lemma \ref{Kullback} and \eqref{tinfty}, together with the fact that $\overline{u}(t)\to p_H^\frac 1 \gamma$ as $n\to\infty$, we then deduce
\begin{equation*}
\norm{u(t) - p_{H}^{1/\gamma}}_{L^{1}(\Omega)}\to 0    \quad \text{when $k\to\infty$}.   
\end{equation*}
Furthermore, by the bound in $L^\infty_tL^{\gamma+1}_x$ given by the control of the energy $\mathcal{E}$ in Proposition \ref{prop:energy/entropy}, we can deduce  the convergence \eqref{conv} by interpolation. The proof of Theorem~\ref{thm} in the case of a nonzero source term $G$ is thus concluded.
\end{proof}

\begin{figure}[h]
    \centering
    \begin{subfigure}[b]{0.45\textwidth}
        \includegraphics[width=\textwidth]{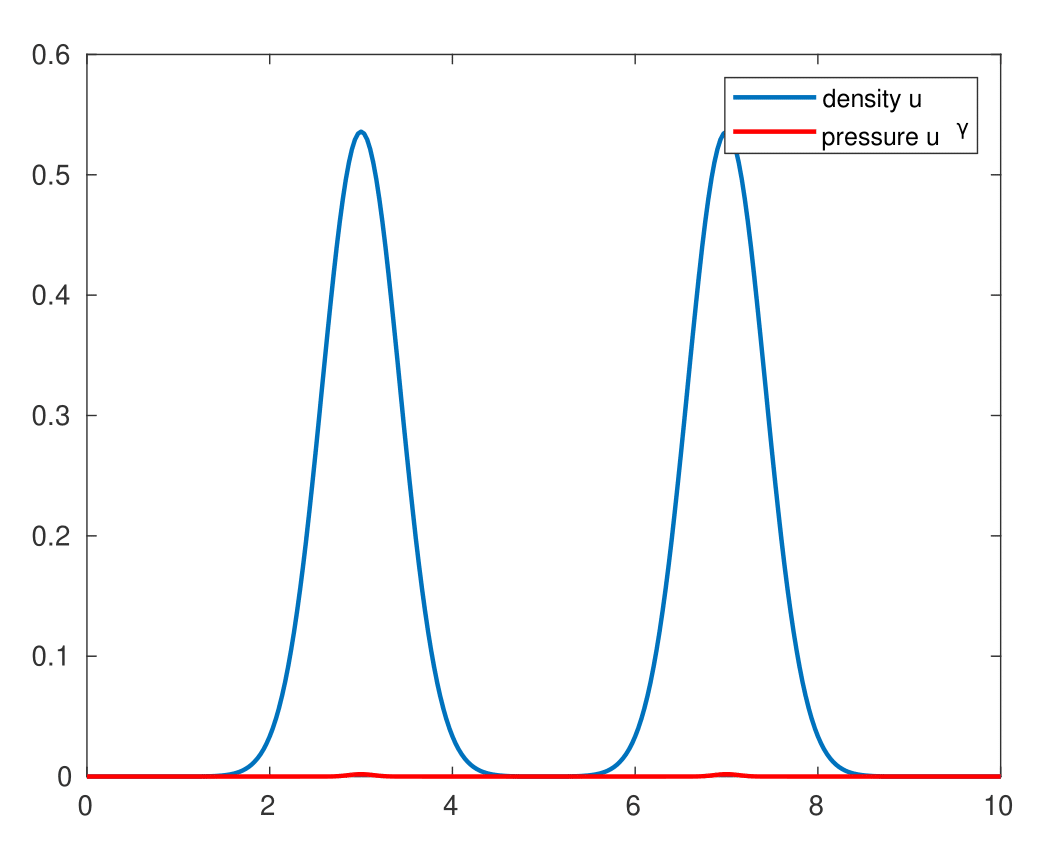}
        \caption{Initial condition}
    \end{subfigure}
    \begin{subfigure}[b]{0.45\textwidth}
        \includegraphics[width=\textwidth]{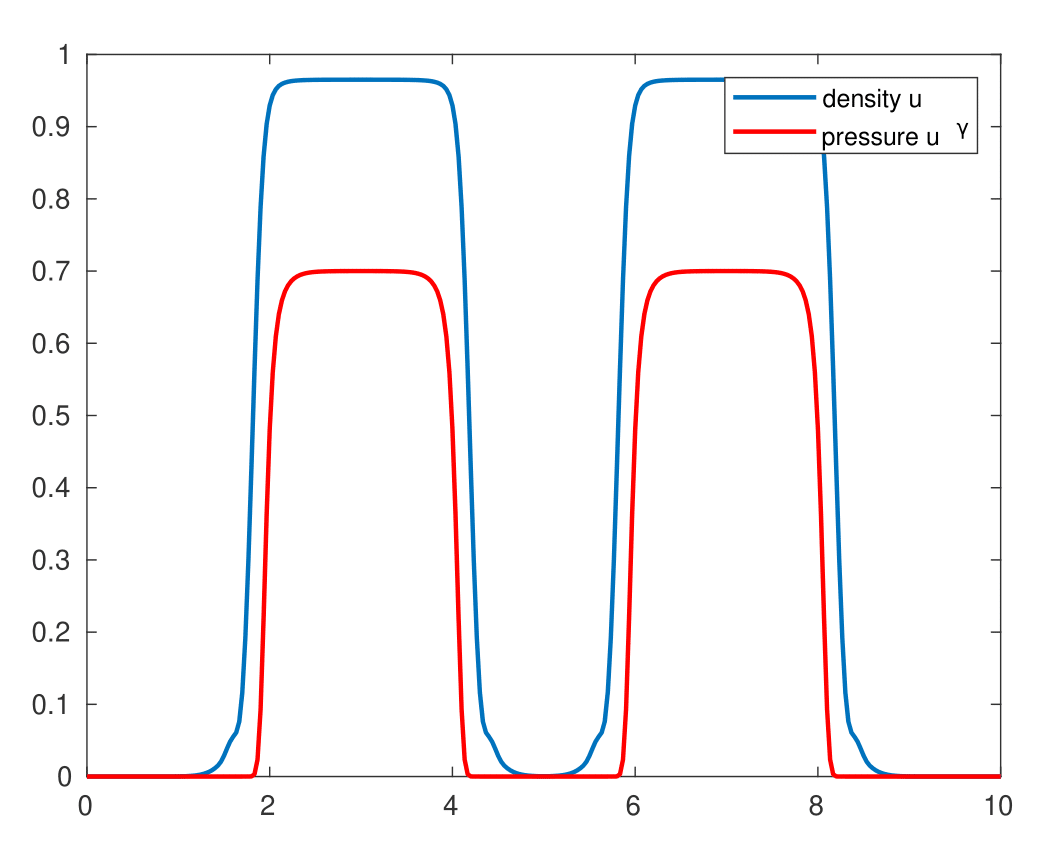}
        \caption{Evolution at $t=0.03$}
    \end{subfigure}
        \begin{subfigure}[b]{0.45\textwidth}
        \includegraphics[width=\textwidth]{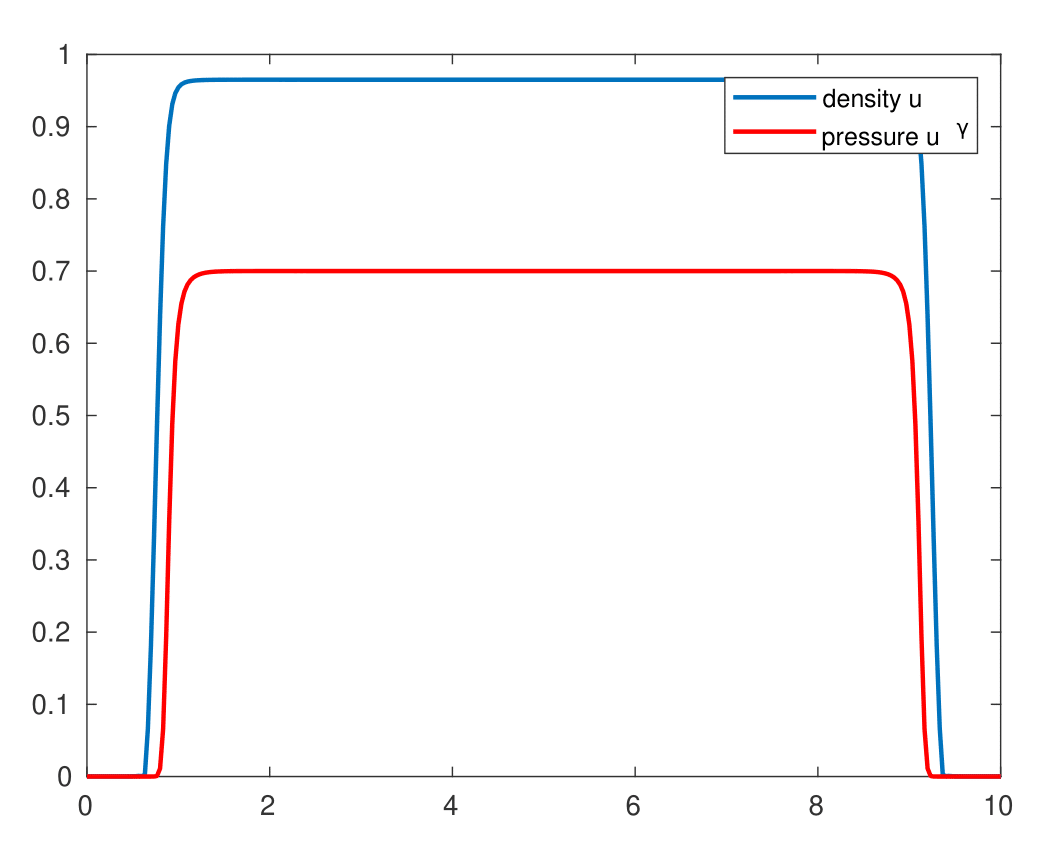}
        \caption{Evolution at $t=0.14$}
    \end{subfigure}
    \begin{subfigure}[b]{0.45\textwidth}
        \includegraphics[width=\textwidth]{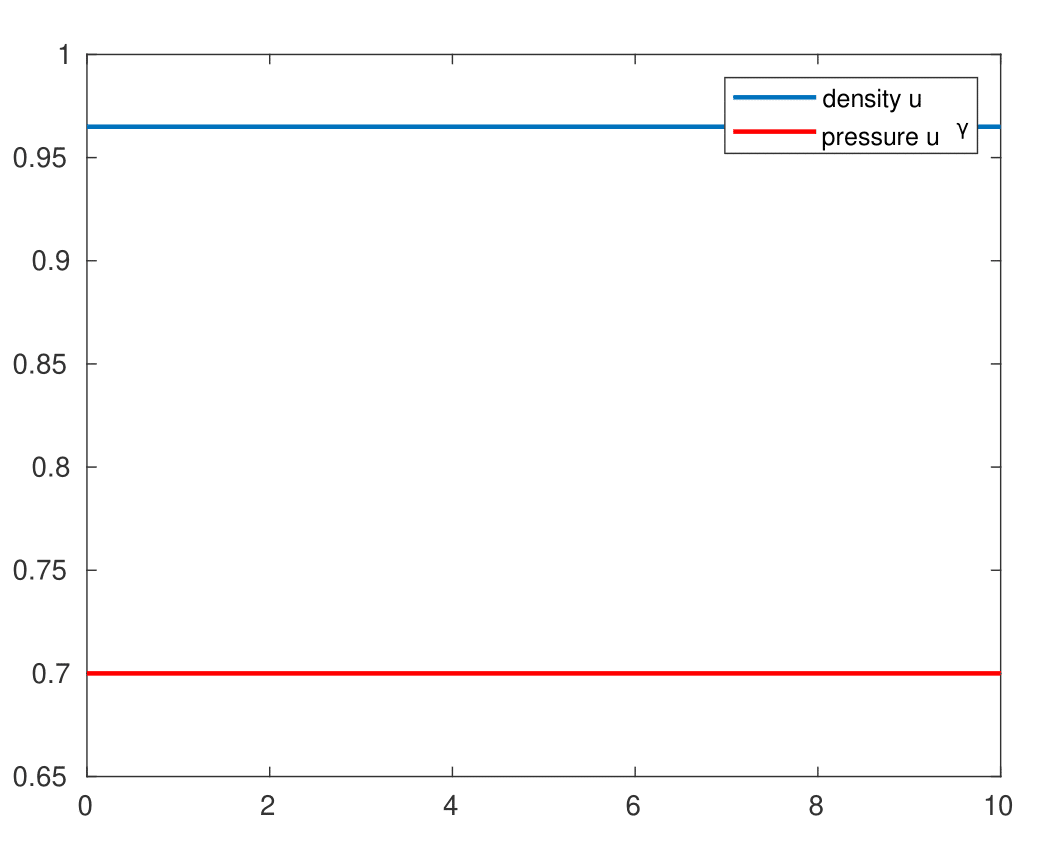}
        \caption{Evolution at $t=0.25$}
    \end{subfigure}
    \caption{We provide here 1D numerical simulations illustrating this theoretical result. We assume $\gamma=10$ and $p_{H}=0.7$. We thus have $p_{H}^{1/\gamma}\approx 0,96$. The initial condition is taken as a double gaussian as in Figure (a).
    }
    \label{fig:asymptotics_source}
\end{figure}


\subsection{Case $G(p)=0$}\label{sect:convergence:stationary_states_g=0}

When there is no source term, the solution converges to the mean value. The argument is a simple consequence of the logarithmic Sobolev inequality and the Csisz\'{a}r–Kullback–Pinsker inequality. We refer to~\cite{MR3770199} for other systems where it is applied. We consider the relative entropy between the solution $u$ and a stationary state $u_{\ast}$, defined as
\begin{equation*}
\Phi(u|u_{\ast})=\int_{\Omega}\left(u\log\left(\frac{u}{u_{\ast}}\right)-u+u_{\ast}\right) \diff x .     
\end{equation*}
In our case $u_{\ast}=\overline{u}_0$ a.e. in $\Omega$. Notice that, by the conservation of mass, we have 
\begin{equation*}
\Phi(u|\overline{u})=\int_{\Omega}u\log\left(\frac{u}{\overline{u}}\right) \diff x . \end{equation*}

Since $\overline{u}(\cdot)\equiv\overline{u}_{0}$, we see that $\Phi(u|\overline{u})$ satisfies the identity (or at least the inequality for weak solutions): for almost any $t\geq0$, 
\begin{equation*}
\frac{d\Phi(u|\overline{u})}{dt} +\frac{1}{2\varepsilon^2} \int_\Omega\int_\Omega \omega_\varepsilon(y)\vert \nabla u(x)-\nabla u(x-y) \vert^2\diff x \diff y+\int_\Omega \dfrac{4\gamma}{(\gamma+1)^2}\left\vert \nabla \vert u\vert^{\frac{\gamma+1}{2}}\right\vert^2 \diff x =0.
\end{equation*}

The generalized logarithmic Sobolev inequality~\cite[Section 3.1]{MR3770199} provides us the exponential decay of the relative entropy. Indeed we have 
\begin{lemma} 
For any $m\ge(d-2)_{+}/d$, there exists $C(\Omega,m)$ such that
\begin{equation*}
\int_{\Omega}\vert \nabla(u^{\frac m 2})\vert^2\diff x \ge C(\Omega,m)\overline{u}^{m-1}
\int_{\Omega}u\log\left(\frac{u}{\overline{u}}\right)\diff x. 
\end{equation*}
\end{lemma}
In our case $m=\gamma+1$ and $\overline{u}\equiv \overline{u}_0$ is bounded from below by a positive constant since we consider an initial condition $u_{0}\not\equiv 0$. Therefore, by the Gronwall Lemma we conclude that the entropy $\Phi$ experiences an exponential decay as $t\to \infty$:
$$
\frac{d\Phi(u|\overline{u})}{dt}+C(\Omega,\gamma)\overline{u}_0^\gamma\Phi(u\vert\overline{u})\leq 0
$$
Therefore, we have, for some $C=C(\Omega,\gamma, \Phi(u_0),\overline{u}_0)$,
$$
\Phi(u|\overline{u})\leq Ce^{-Ct},\quad\forall t\geq 0.
$$
To prove that this implies the exponential decay of the solution we use the Csisz\'{a}r–Kullback–Pinsker inequality of Lemma \ref{Kullback}: there exists $C(\Omega)$ such that
	\begin{equation*}
	\Phi(u|\overline{u})\ge C(\Omega)\|u-\overline{u}_0\|^{2}_{L^{1}(\Omega)}.  
	\end{equation*}
The exponential decay in $L^1(\Omega)$ of $u$ towards $\overline{u}_0$ then easily follows. By the $L^\infty(0,T; L^{\gamma+1}(\Omega))$ bound given by the control of the energy $\mathcal{E}$ (which is the same as in the case with a source term $G$ given in Proposition \ref{prop:energy/entropy}), we can in conclusion deduce the exponential convergence \eqref{conv2} by interpolation. This ends the proof of Theorem~\ref{thm:conv_stationary}.

\subsection{Longtime behavior of the local Cahn-Hilliard equation}
The nonlocal Cahn-Hilliard equation can be also seen as an approximation of the local Cahn-Hilliard equation:
\begin{equation}
\label{local}
\begin{cases}
\partial_t u -\text{div}(u\nabla\mu)=uG(p)\quad \text{in }\Omega\times (0,T),\\
\mu=p-\Delta u,\quad p=u^{\gamma},\\
u(0)=u_0\quad \text{ in }\Omega.
\end{cases}
\end{equation}

This follows, at least formally, with a Taylor expansion, using the symmetry of the kernel $\omega_{\varepsilon}$ in the operator $B_{\eps}$ of~\eqref{nonlocal}, and for a rigorous proof we refer, e.g., to~\cite{elbar-skrzeczkowski}. Therefore, one may wonder whether the previous results obtained for the nonlocal Cahn-Hilliard equation also hold for the local one. It turns out that for the convergence to the stationary states, the result is the same. Indeed, one mainly uses arguments based on the entropy, so that the nonlocal term does not play a role: concerning the entropy $\Phi$, we can consider again \eqref{eq:entropy} and formally get for \eqref{local}
\begin{align}
\frac{d}{dt}\Phi(u)+\frac{1}{p_H^{\frac 1 \gamma}}\int_\Omega \vert \Delta u\vert^2\diff x +\frac{1}{p_H^{\frac 1 \gamma}}\int_\Omega \dfrac{4\gamma}{(\gamma+1)^2}\left\vert \nabla \vert u\vert^{\frac{\gamma+1}{2}}\right\vert^2 \diff x -\int_\Omega u\log\left(\frac{u}{p_H^{\frac 1 \gamma}}\right)G(p)\diff x =0,
\label{entrid}
\end{align}
which is very similar to the result in Proposition \ref{prop:energy/entropy}.
Concerning the energy, we set 
$$
\mathcal{E}_l(u):=\frac{1}{2}\int_\Omega \vert \nabla u\vert^2\diff x +\int_\Omega \dfrac{u^{1+\gamma}}{1+\gamma}\diff x 
$$
and thus
\begin{align*}
\frac{d}{dt}\mathcal{E}_l(u)+\int_\Omega u\vert \nabla \mu\vert^2\diff x =\int_\Omega uG(p)\mu \diff x .
\end{align*}
Now we observe that, integrating by parts,
\begin{align*}
\int_\Omega uG(p)\mu \diff x &=\int_\Omega u(p_H-p)p\diff x -\int_\Omega u(p_H-p)\Delta u \diff x  
\\
& =\int_\Omega u(p_H-p)p\diff x  +\int_\Omega\vert\nabla u\vert^2(p_H-p)\diff x -\int_\Omega u\nabla p\cdot \nabla u \diff x \\&=  
\int_\Omega u(p_H-p)p\diff x  +\int_\Omega\vert\nabla u\vert^2(p_H-p)\diff x -\gamma\int_\Omega p\vert\nabla u\vert^2 \diff x .
\end{align*}
Moreover, notice that, since the relation \eqref{mass} still holds with the same proof,
$$
p_H\int_\Omega up\diff x \leq \frac{1}{2}\int_\Omega up^2\diff x +\frac{p_H^2 }{2}\int_\Omega u\diff x \leq \frac{1}{2}\int_\Omega up^2\diff x +C.
$$
Then we can rewrite the energy inequality as 
\begin{align}
&\nonumber\frac{d}{dt}\mathcal{E}_l(u)+\int_\Omega u\vert \nabla \mu\vert^2\diff x +\frac 1 2\int_\Omega p^2 u\diff x +\int_\Omega \vert\nabla u\vert^2p\diff x +\gamma \int_\Omega p\vert \nabla u\vert^2 \diff x \\&\leq \int_\Omega \vert\nabla u\vert^2\diff x +C\leq C(\mathcal{E}_l(u)+1),\nonumber
\end{align}
which is again very similar to the one obtained in Proposition \ref{prop:energy/entropy} for the nonlocal case.
Therefore, with these estimates we can basically perform again all the arguments of Section \ref{section::equil} and obtain again the same result as in Theorem \ref{thm:conv_stationary}. Note that in the local case, differently from the nonlocal one, we can also repeat the same arguments in the case of a smooth bounded domain $\Omega\subset \R^d$ with homogeneous Neumann boundary conditions $u\nabla\mu\cdot \textbf{n}=0$ and $\nabla u\cdot \textbf{n}=0$ on $\partial\Omega\times(0,\infty)$, where $\textbf{n}$ is the outward unit normal. 
\begin{remark}
The incompressible limit, $\gamma\to\infty$, is very different and remains  an open question in the local Cahn-Hilliard case. Indeed, obtaining an equation for the pressure $p$ from which to deduce a uniform-in-$\gamma$ $L^2(0,T;H^1(\Omega))$-control on $p$ seems still out of reach. Therefore no analogous of Theorem \ref{thm:incompressible}  can be stated in this local case.
\end{remark}

\section*{Acknowledgements}
J.S. was supported by the National Science Center grant 2017/26/M/ST1/00783. A.P. has been partially funded by MIUR-PRIN research grant n. 2020F3NCPX and is also member of Gruppo Nazionale per l'Analisi
Ma\-te\-ma\-ti\-ca, la Probabilit\`{a} e le loro Applicazioni (GNAMPA),
Istituto Nazionale di Alta Matematica (INdAM).

\appendix

\section{Technical tools} \label{app:technical_tools}

Several tools have been used to carry out some proofs. First, we present a lemma about geometric convergence of numerical sequences, whose proof can be easily obtained by induction (see, e.g., \cite[Ch.2, Lemma 5.6]{MR0241822} ): 
\begin{lemma}
	\label{conva}
	Let $\{y_n\}_{n\in\N\cup \{0\}}\subset \R^+$ satisfy the recursive inequality 
	\begin{align}
	y_{n+1}\leq Cb^ny_n^{1+\epsilon},
	\label{ineq}\quad \forall n\geq 0,  \qquad \text{and} \qquad  y_0\leq \theta:= C^{-\frac{1}{\epsilon}}b^{-\frac{1}{\epsilon^2}},
	\end{align}
	for some $C>0$, $b>1$ and $\epsilon>0$. Then, $y_n\to 0$ for $n\to \infty$ with geometric rate
	\begin{align}
	y_n\leq \theta b^{-\frac{n}{\epsilon}},\qquad \forall n\geq 0.
	\label{yn}
	\end{align}
\end{lemma}

Next, we state a theorem concerning the absolute continuity of some integrals of convex functions in $\R$, whose proof can be found, e.g. in~\cite[p.101]{MR610796}:
\begin{theorem}
Let $T>0$ and let $h: \R\to\R$ be a convex and lower semicontinuous function. Assume that 
\begin{itemize}
	\item $f\in L^2(0,T;H^1(\Omega))$ and $\partial_t f\in L^2(0,T; (H^1(\Omega))')$,
	\item $g(x,t)\in \partial h(x,t)$ for almost every $(x,t)\in \Omega\times(0,T)$,
	\item $g\in L^2(0,T;H^1(\Omega))$.
\end{itemize} 
Then, the function $t\mapsto \int_\Omega h(f(x,t))$ is absolutely continuous on $[0,T]$ and, 
$$
\frac d {dt} \int_\Omega h(f)\diff x =\langle \partial_t f, g\rangle \quad \text{for almost any } t\in(0,T).
$$ 	
\label{Har}
\end{theorem}


We then propose a control on the $H^1(\Omega)$-norm related to the use of $\omega_\varepsilon$. 
\begin{lemma}
\label{Poincar}
    There exists $\varepsilon_0>0$ and a constant $C$ such that for $\varepsilon \in (0, \varepsilon_0)$
and all $f\in L^2(\Omega)$ we have
\begin{align}\label{eq:poincare_L2}
\Vert f - \overline{f} \Vert_{L^2(\Omega)}^2\leq  \frac{C}{2\varepsilon^2} \int_\Omega\int_\Omega \omega_\varepsilon(y)\vert  f(x)- f(x-y) \vert^2\diff x \diff y,
\end{align}
where $\overline{f}$ is the average of $f$ over $\Omega$. Similarly, for all $\alpha$, there exists $\varepsilon_0(\alpha)>0$ and constant $C(\alpha)$ such that for all $\varepsilon \in (0, \varepsilon_0)$
and all $f\in H^1(\Omega)$ we have
\begin{align}
\Vert f\Vert_{H^1(\Omega)}^2\leq  \frac{\alpha}{2\varepsilon^2} \int_\Omega\int_\Omega \omega_\varepsilon(y)\vert \nabla f(x)-\nabla f(x-y) \vert^2\diff x \diff y+C(\alpha) \Vert f \Vert_{L^1(\Omega)}^2.
\label{p_p}
\end{align}

\end{lemma}
\begin{proof}
The proof is identical to the one in \cite[Lemma C.3]{elbar-skrzeczkowski} by substituting the norm $\Vert~ \cdot~\Vert_{L^2(\mathbb{T}^d)}$ with the norm $\Vert\cdot\Vert_{L^1(\mathbb{T}^d)}$. Indeed, with the notation of the proof of that Lemma, also $n\Vert g_n\Vert_{L^1(\mathbb{T}^d)}<1$ implies that the limit function $g=0$, exactly as in the case $L^2(\mathbb{T}^d)$. 
\end{proof}

In conclusion, we recall the Csisz\'{a}r–Kullback–Pinsker inequality (see, e.g., \cite{Bolley_villani05}), which is  essential to study the  asymptotic behavior of weak solutions
\begin{lemma}
\label{Kullback}
	For any non-negative $u\in L^1(\Omega)$ 
	\begin{equation*}
	4 |\Omega| \overline{u} \int_\Omega u\log\left(\dfrac{u}{\overline{u}}\right)\diff x \ge \|u-\overline{u}\|^{2}_{L^{1}(\Omega)}.    
	\end{equation*}
\end{lemma}

\bibliographystyle{abbrv}
\bibliography{fastlimit}

\end{document}